\documentclass[12pt]{amsart}
\usepackage{graphicx} 
\usepackage{amsmath, amsthm, amssymb}
\usepackage{thmtools, thm-restate}
\usepackage{textcomp}
\usepackage{xcolor}
\usepackage{color-edits}
\usepackage[margin=1in]{geometry}
\usepackage{algorithm}
\usepackage[noend]{algpseudocode}
\usepackage{float}
\usepackage{multicol}
\usepackage{caption}
\usepackage[hypertexnames=false,hidelinks]{hyperref}
\usepackage{booktabs}
\usepackage{footnote}
\usepackage{float}
\usepackage{color-edits}
\usepackage{algorithm}
\usepackage[noend]{algpseudocode}
\usepackage{multirow}
\usepackage{array}
    \newcolumntype{P}[1]{>{\centering\arraybackslash}p{#1}}
    \newcolumntype{M}[1]{>{\centering\arraybackslash}m{#1}}
\usepackage{bm}
\usepackage{graphicx}
\usepackage{graphbox}

\usepackage{makecell}
\usepackage{subcaption}
\usepackage[inline]{enumitem}
\usepackage{adjustbox,lipsum}

\usepackage{tikz}
\usetikzlibrary{calc,fit,backgrounds}
\usepackage{cases}

\addauthor[Kartik]{kl}{blue}
\addauthor[Laura]{lm}{red}
\addauthor[Kelly]{ki}{orange}
\addauthor[Fabrizio]{fp}{blue}

\def\BibTeX{{\rm B\kern-.05em{\sc i\kern-.025em b}\kern-.08emT\kern-.1667em\lower.7ex\hbox{E}\kern-.125emX}}

\def\comellasdate{October 2026}

\newcommand{\F}[0]{\mathbb{F}}

\newtheorem{theorem}{Theorem}[section]

\newtheorem{lemma}[theorem]{Lemma}

\newtheorem{corollary}[theorem]{Corollary}

\newtheorem{definition}[theorem]{Definition}
\newtheorem{proposition}[theorem]{Proposition}

\title{Kronecker Products, Polarity Quotients and Large Graph Constructions}
\author{Kelly Isham}
\address{Colgate University, Hamilton, NY}
\email{kisham@colgate.edu}
\author{Kartik Lakhotia}
\address{Intel, Santa Clara, CA}
\email{kartik.lakhotia@intel.com}
\author{Laura Monroe}
\address{Los Alamos National Laboratory, Los Alamos, NM}
\email{lmonroe@lanl.gov}
\author{Fabrizio Petrini}
\address{Intel, Santa Clara, CA}
\email{fabrizio.petrini@intel.com}
\date{}

\begin{document}

\begin{abstract}
    In this paper, we establish a structural compatibility between the Kronecker product  of bipartite graphs that admit polarity and their polarity quotients, and provide a sharp upper bound on the diameter of these graphs. We use these properties to construct new families of graphs, all of which have orders that asymptotically approach the Moore bound.
   
   We show that the bipartite compatibility graphs have a natural polarity, and that the corresponding polarity quotients comprise a new family of graphs  of any even diameter $2k$ with  degree $(q+1)^k$ for any prime power $q$, much smaller than those of other known families of that diameter and field order $q$. These families generalize the Erd{\H o}s-R\'enyi-Brown polarity graphs. Kronecker products of these polarity quotients give another new family of graphs  of any even diameter for every pair of prime powers $q,r$.
   
   Certain generalized polygons are known  to admit polarity with polarity quotients approaching the Moore bound asymptotically. Kronecker products of these polarity quotients give another new family of graphs  of diameters $2$, $3$ and $5$, including
    a diameter-$3$ degree-$18$ graph of order larger than any previously known for that degree and diameter. 
\end{abstract}
\maketitle
\section{Introduction }\label{sec:intro_contributions}
\subsection{Motivation and Results}

In this paper, we address the problem of finding large graphs of a given degree and diameter, known as the \emph{degree-diameter problem}, with especial emphasis on graphs that approach the \emph{Moore bound}, the mathematical upper bound on graph order for a given degree and diameter.

We have constructed two different new families of graphs  whose orders approach the Moore bound for every even diameter, one of which generalizes the well-known diameter-$2$ Erd{\H o}s-R\'enyi-Brown polarity graphs \cite{erdosrenyi1962, brown_1966} to graphs of any even diameter by taking polarity quotients of bipartite compatibility graphs. This family has degrees asymptotically as small as the fourth root of those in the related halved-flag graphs discussed in \cite{vanbatenburg2026asymptoticallyattainingmoorebound} at the same diameter and field order. 

The second family uses Kronecker products of the polarity quotients of the bipartite compatibility graphs and in so doing, broadens the range of possible degrees for any even diameter.
We have also constructed new families of graphs  using a similar Kronecker construction of diameters $\{2,3,5\}$  based on generalized polygons. 

We accomplish the Kronecker construction in part through structural theorems showing compatibility between the Kronecker product on bipartite graphs and their polarity quotients, and then establishing exact diameters for certain graphs and a multiplicative rule for the asymptotic fractions of the Moore bound.

\subsection{Three New Large Constructions}
We show three new constructions in this section that approach the Moore bound, two based on the bipartite compatibility graphs of \cite{vanbatenburg2026asymptoticallyattainingmoorebound} and the other based on generalized polygons.

\subsubsection{Polarity Quotients of Bipartite Compatibility Graphs (Diameter $2k$)}
We show that the bipartite compatibility graphs $\mathbb{C}(k,q)$ \cite{vanbatenburg2026asymptoticallyattainingmoorebound} admit a natural polarity, and that their polarity quotients approach the Moore bound asymptotically for every even diameter $2k$. 
These graphs have the added advantage that their degrees are much smaller than the degrees of the even-diameter graphs in \cite{vanbatenburg2026asymptoticallyattainingmoorebound}: for the same diameter $2k$ and field order $q$, they have exact degree $(q+1)^k$, asymptotically the $4^{th}$ root of the degrees of the even-diameter $H_{2k,q}$ graphs from \cite{vanbatenburg2026asymptoticallyattainingmoorebound}, which are asymptotic to $(q+1)^{4k}$. We thus derive the following theorem:
\begin{theorem} \label{th:main_polarity_bcg}
    Let $q$ be a prime power and let $k \in \mathbb{Z}_{>1}.$ Then the bipartite compatibility graph $\mathbb{C}(k,q)$ admits a polarity $\beta$, and its polarity quotient graph $\overline{P_\beta(\mathbb{C}(k,q))}$ has diameter $2k$ and degree $(q+1)^k.$ The orders of the graphs $\overline{P_\beta(\mathbb{C}(k,q))}$ are  $\frac{[2k+1]_q!}{(q+1)^k}$ and approach the Moore bound as $q \rightarrow \infty.$
\end{theorem}

 These graphs are discussed in Section~\ref{sec:bipartite_compatibility}.

\subsubsection{Kronecker Products of Polarity Quotients of Bipartite Compatibility Graphs (Diameter $2k$)}
The diameter results on the polarity quotients of the bipartite compatibility graphs allow us to use the theoretical results in Section~\ref{sec:main_bipartite} to construct Kronecker-product graphs that approach the Moore bound asymptotically for every even diameter $2k$, by using polarity quotients of the bipartite compatibility graphs as factors. These products have degree $(q+1)^k(r+1)^k$, where $q$ and $r$ are prime powers, and have degree $(q+1)^{2k}$ when $q=r$, asymptotically the square root of the degrees of the $H_{2k,q}$ graphs from \cite{vanbatenburg2026asymptoticallyattainingmoorebound}. Also, since these allow any pair $q,r$ of prime powers, the set of admissible degrees is larger than that in \cite{vanbatenburg2026asymptoticallyattainingmoorebound}. 
This is discussed in Section~\ref{sec:bipartite_compatibility}.

\begin{theorem}\label{th:main_bcg_kronecker_moore}
    Let $q$ and $r$ be prime powers and let $k \in \mathbb{Z}_{>1}.$ Then the polarity quotients of Reduced-Kronecker products of bipartite compatibility graphs $\overline{P_{\beta}(\mathbb{C}(k,q)\odot \mathbb{C}(k,r))}$ form an infinite family of graphs each having diameter $2k$ and degree $(q+1)^k(r+1)^k.$ The orders of the graphs $\overline{P_{\beta}(\mathbb{C}(k,q)\odot \mathbb{C}(k,r))}$ are $\frac{[2k+1]_q! \cdot [2k+1]_r!}{(q+1)^k(r+1)^k}$ and approach the Moore bound as $q, r\rightarrow \infty.$ \end{theorem}

\subsubsection{Kronecker Products of Polarity Quotients of Generalized Polygons (Diameters $2,3,5$)}
Generalized polygons $\mathbb{G}_n(q,q)$ are known to admit polarity when $n\in \{2,3,4,6\}$ and $q$ is an integer satisfying certain conditions  that depend on $n$.
    Polarity quotients of non-digon generalized polygons also satisfy the theorems and corollaries in Section~\ref{sec:main_bipartite}, so may be used as factors in Kronecker products for another set of graph families that  approach the Moore bound asymptotically for diameters in $\{2,3,5\}.$  We discuss this in Section~\ref{sec:genpolys_complete}, and address other types of generalized polygons not always treated in the existing literature, such as digons and thin generalized polygons. Using this construction we get a degree-$18$ diameter-$3$ graph that is the largest known at time of writing, with $2340$ vertices.

\subsection{Kronecker Products and Polarity Quotients: Structure, Diameter and Order}
Here and throughout the paper, the \emph{Kronecker}, or \emph{direct}, product of graphs $H$ and $K$ has vertex set $H \times K$ and edges $((h_1, k_1) ,(h_2,k_2))$ if and only if $(h_1, h_2) \in E_H$ and $(k_1, k_2) \in E_K$. This product is denoted $H \otimes K$. We assume that the factor graphs $H$ and $K$ are undirected and connected with no multiedges, but we allow loops. A \emph{polarity} on a bipartite graph $\mathbb{G}(U = (A,B), \mathbb{E})$ is an involutional automorphism $\rho$ that reverses the partition. $\mathbb{H} \odot \mathbb{K}$ denotes one of the two isomorphic components of $\mathbb{H} \otimes \mathbb{K}$ (as defined in Section~\ref{sec:pautomorphisms_kronecker}). Finally, when we use the term \emph{degree of a graph}, we mean the maximum degree of all the vertices.

\subsubsection{Structural Result}
We begin with the isomorphism that is the basis for the theoretical results in this paper. Theorem~\ref{th:main_isomorphism} is closely related to earlier factorization results for bipartite graphs with bipartition-reversing involutions from \cite{bipartition_reversing_involutions_10}. Its additional content is the compatibility of these factorizations with the specified polarities and their composed polarity on the product.

\begin{theorem}

\label{th:main_isomorphism}
        Let $\mathbb{H}(U_H=(A_H,B_H),\mathbb{E}_H)$ and $\mathbb{K}(U_K=(A_K,B_K),\mathbb{E}_K)$ be connected bipartite graphs that admit polarities $\rho_H$ and $\rho_K$, with polarity quotient graphs $P_H(\mathbb{H})$ and $P_K(\mathbb{K})$, and let $P(\mathbb{H} \odot \mathbb{K})$ be  the polarity quotient graph on $\mathbb{H} \odot \mathbb{K}$ with respect to the composed polarity $\rho = (\rho_H, \rho_K)$. Then $P(\mathbb{H} \odot \mathbb{K})$ is isomorphic to $P_H(\mathbb{H}) \otimes P_K(\mathbb{K}).$  
        \end{theorem}

Corollary~\ref{cor:main_isomorphism} then gives the polarity quotient of the entire Kronecker product:
$$
    P(\mathbb{H} \otimes \mathbb{K}) \cong (P_H(\mathbb{H}) \otimes P_K(\mathbb{K})) \sqcup (P_H(\mathbb{H}) \otimes P_K(\mathbb{K})).
$$

In effect, Theorem~\ref{th:main_isomorphism} and Corollary~\ref{cor:main_isomorphism} show that the Kronecker product is compatible with the polarity quotient 
on the class of bipartite graphs with polarity, as seen in Figure~\ref{fig:commutation_new}.  
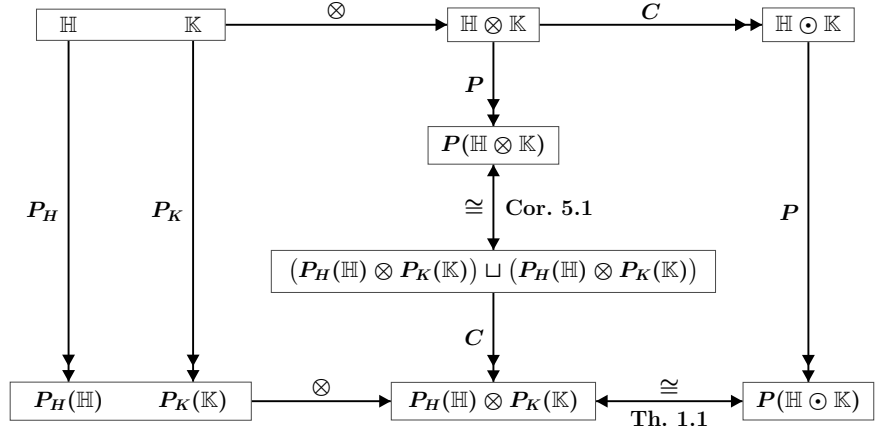
\begin{figure}[htbp]
\centering

\resizebox{.7\textwidth}{!}{%

\begin{tikzpicture}[
    x=1cm,
    y=1cm,
    line cap=butt,
    line join=miter,
    every node/.style={
        font=\Large\boldmath,
        align=center
    },
    box/.style={
        draw=black!70,
        fill=white,
        line width=0.8pt,
        inner xsep=10pt,
        inner ysep=6pt
    },
    arrowlabel/.style={
        font=\Large\boldmath,
        inner sep=2pt
    },
    reflabel/.style={
        font=\Large,
        inner sep=2pt
    }
]


\def\arrowlinewidth{1.5pt}

\def\RightHead#1{%
    \fill
        (#1)
        -- ++(-9pt,5pt)
        -- ++(0pt,-10pt)
        -- cycle;
}

\def\LeftHead#1{%
    \fill
        (#1)
        -- ++(9pt,5pt)
        -- ++(0pt,-10pt)
        -- cycle;
}

\def\DownHead#1{%
    \fill
        (#1)
        -- ++(-5pt,9pt)
        -- ++(10pt,0pt)
        -- cycle;
}

\def\UpHead#1{%
    \fill
        (#1)
        -- ++(-5pt,-9pt)
        -- ++(10pt,0pt)
        -- cycle;
}


\coordinate (Hpos)  at (0,10.3);
\coordinate (Kpos)  at (3.4,10.3);
\coordinate (PHpos) at (0,0);
\coordinate (PKpos) at (3.4,0);

\coordinate (topcenter)    at (11.6,10.3);
\coordinate (uppercenter)  at (11.6,7.0);
\coordinate (middlecenter) at (11.6,3.55);
\coordinate (bottomcenter) at (11.6,0);

\coordinate (topright)    at (20.2,10.3);
\coordinate (bottomright) at (20.2,0);


\node[inner sep=0pt] (H) at (Hpos)
    {$\mathbb H$};

\node[inner sep=0pt] (K) at (Kpos)
    {$\mathbb K$};

\node[inner sep=0pt] (PH) at (PHpos)
    {$P_H(\mathbb H)$};

\node[inner sep=0pt] (PK) at (PKpos)
    {$P_K(\mathbb K)$};


\begin{scope}[on background layer]

\node[
    draw=black!70,
    fill=white,
    line width=0.8pt,
    inner xsep=18pt,
    inner ysep=6pt,
    fit=(H)(K)
] (HK) {};

\node[
    draw=black!70,
    fill=white,
    line width=0.8pt,
    inner xsep=18pt,
    inner ysep=6pt,
    fit=(PH)(PK)
] (PHPK) {};

\end{scope}


\node[box] (HxK) at (topcenter)
    {$\mathbb H\otimes\mathbb K$};

\node[box] (HodotK) at (topright)
    {$\mathbb H\odot\mathbb K$};

\node[box] (PHxK) at (uppercenter)
    {$P(\mathbb H\otimes\mathbb K)$};

\node[
    box,
    inner xsep=13pt
] (Union) at (middlecenter)
    {$
      \bigl(
          P_H(\mathbb H)
          \otimes
          P_K(\mathbb K)
      \bigr)
      \sqcup
      \bigl(
          P_H(\mathbb H)
          \otimes
          P_K(\mathbb K)
      \bigr)
    $};

\node[
    box,
    inner xsep=13pt
] (PHtensorPK) at (bottomcenter)
    {$
      P_H(\mathbb H)
      \otimes
      P_K(\mathbb K)
    $};

\node[box] (PHodotK) at (bottomright)
    {$P(\mathbb H\odot\mathbb K)$};


\draw[line width=\arrowlinewidth]
    (HK.east) -- (HxK.west);

\RightHead{HxK.west}

\node[arrowlabel]
    at ($(HK.east)!0.5!(HxK.west)+(0,0.38)$)
    {$\otimes$};

\draw[line width=\arrowlinewidth]
    (HxK.east) -- (HodotK.west);

\RightHead{HodotK.west}
\RightHead{$(HodotK.west)+(-12pt,0pt)$}

\node[arrowlabel]
    at ($(HxK.east)!0.5!(HodotK.west)+(0,0.38)$)
    {$C$};


\coordinate (PHstart) at (H.center |- HK.south);
\coordinate (PHend)   at (PH.center |- PHPK.north);

\draw[line width=\arrowlinewidth]
    (PHstart) -- (PHend);

\DownHead{PHend}
\DownHead{$(PHend)+(0pt,12pt)$}

\node[arrowlabel,left=5pt]
    at ($(PHstart)!0.5!(PHend)$)
    {$P_H$};

\coordinate (PKstart) at (K.center |- HK.south);
\coordinate (PKend)   at (PK.center |- PHPK.north);

\draw[line width=\arrowlinewidth]
    (PKstart) -- (PKend);

\DownHead{PKend}
\DownHead{$(PKend)+(0pt,12pt)$}

\node[arrowlabel,left=5pt]
    at ($(PKstart)!0.5!(PKend)$)
    {$P_K$};


\draw[line width=\arrowlinewidth]
    (HxK.south) -- (PHxK.north);

\DownHead{PHxK.north}
\DownHead{$(PHxK.north)+(0pt,12pt)$}

\node[arrowlabel,left=5pt]
    at ($(HxK.south)!0.5!(PHxK.north)$)
    {$P$};

\draw[line width=\arrowlinewidth]
    (PHxK.south) -- (Union.north);

\UpHead{PHxK.south}
\DownHead{Union.north}

\node[arrowlabel,left=5pt]
    at ($(PHxK.south)!0.5!(Union.north)$)
    {$\cong$};

\node[reflabel,right=7pt]
    at ($(PHxK.south)!0.5!(Union.north)$)
    {%
        \hyperref[cor:main_isomorphism]
        {\textbf{Cor.~\ref*{cor:main_isomorphism}}}%
    };

\draw[line width=\arrowlinewidth]
    (Union.south) -- (PHtensorPK.north);

\DownHead{PHtensorPK.north}
\DownHead{$(PHtensorPK.north)+(0pt,12pt)$}

\node[arrowlabel,left=5pt]
    at ($(Union.south)!0.5!(PHtensorPK.north)$)
    {$C$};


\draw[line width=\arrowlinewidth]
    (HodotK.south) -- (PHodotK.north);

\DownHead{PHodotK.north}
\DownHead{$(PHodotK.north)+(0pt,12pt)$}

\node[arrowlabel,left=5pt]
    at ($(HodotK.south)!0.5!(PHodotK.north)$)
    {$P$};


\draw[line width=\arrowlinewidth]
    (PHPK.east) -- (PHtensorPK.west);

\RightHead{PHtensorPK.west}

\node[arrowlabel]
    at ($(PHPK.east)!0.5!(PHtensorPK.west)+(0,0.38)$)
    {$\otimes$};

\draw[line width=\arrowlinewidth]
    (PHtensorPK.east) -- (PHodotK.west);

\RightHead{PHodotK.west}
\LeftHead{PHtensorPK.east}

\node[arrowlabel]
    at ($(PHtensorPK.east)!0.5!(PHodotK.west)+(0,0.38)$)
    {$\cong$};

\node[reflabel]
    at ($(PHtensorPK.east)!0.5!(PHodotK.west)+(0,-0.48)$)
    {%
        \hyperref[th:main_isomorphism]
        {\textbf{Th.~\ref*{th:main_isomorphism}}}%
    };

\end{tikzpicture}
}
\caption{The commutative diagram of Theorem~\ref{th:main_isomorphism}. Arrows labeled $\otimes$ denote the Kronecker product, while all other arrows are graph homomorphisms. Double arrows denote surjective homomorphisms, each reducing graph order by a factor of $2$. $P, P_H, P_K$ are polarity quotients, and $C$ is a homomorphism collapsing two isomorphic components into one.}

\label{fig:commutation_new}
\end{figure}

\subsubsection{Diameter Results}
The theorems and corollaries in this subsection follow from Theorem~\ref{th:main_isomorphism}, by linking previously known inequalities on $P(\mathbb{H} \odot \mathbb{K})$ and $P_H(\mathbb{H}) \otimes P_K(\mathbb{K}).$
 Theorem~\ref{th:ulb} 
 gives an upper and lower bound on the diameter $D$ of these isomorphic graphs, in terms of the diameters of $\mathbb{H}$ and $\mathbb{K}$ and diameters of their polarity quotients $P_H(\mathbb{H})$ and $P_K(\mathbb{K})$.
\begin{theorem}
\label{th:ulb}
    Let $\mathbb{H}$ and $\mathbb{K}$ be as in Theorem~\ref{th:main_isomorphism} with diameters $D(\mathbb{H})$ and $D(\mathbb{K}),$ and let their polarity quotient graphs have diameter $D(P_{H}(\mathbb{H}))$ and $D(P_{K}(\mathbb{K})).$  Then the isomorphic graphs $P_{H}(\mathbb{H}) \otimes P_{K}(\mathbb{K})$ and $P(\mathbb{H} \odot \mathbb{K})$ have diameter $D,$ where 
    $$
    \max(D(P_{H}(\mathbb{H})),D(P_{K}(\mathbb{K}))) \le D\le \max(D(\mathbb{H}),D(\mathbb{K}))-1.
    $$
\end{theorem}

Corollary~\ref{cor:ulb} then gives an exact diameter for the isomorphic graphs in Theorem~\ref{th:main_isomorphism} when the diameters of the parent graphs and their polarity quotients differ by $1$. Any pair of graphs satisfying the conditions of Corollary~\ref{cor:ulb} may be used to obtain new graphs whose exact diameter is known. Bipartite compatibility graphs and generalized polygons both meet these conditions, and we exploit this fact to devise the new constructions in this paper.

\begin{corollary}\label{cor:ulb}
    Let $\mathbb{H}$ and $\mathbb{K}$ be as in Theorem~\ref{th:ulb}, where $D(P_H(\mathbb{H})) = D(\mathbb{H})-1$ and $D(P_K(\mathbb{K})) = D(\mathbb{K})-1.$ Then the isomorphic graphs $P_H(\mathbb{H}) \otimes P_K(\mathbb{K})$ and $P(\mathbb{H} \odot \mathbb{K})$ have diameter $\max(D(\mathbb{H}),D(\mathbb{K}))-1.$
\end{corollary}

The operations here extend to an arbitrary number of bipartite graphs with polarity, so Theorems~\ref{th:main_isomorphism} and \ref{th:ulb} can be extended to multiple such graphs. 

\subsubsection{Multiplicative Rule for Moore-Bound Fractions}
These theorems also give a multiplicative rule for the fractions of the Moore bound approached by families of polarity quotients. 
Under the hypotheses of Theorem~\ref{th:new_general_MB_fraction_new}, the product has degree and order equal to the product of the factor degrees and orders, and retains the same diameter. Since MB($\Delta, D) \sim \Delta^D,$ the asymptotic Moore-bound fractions multiply when the factor degrees tend to infinity.

We note that the degree condition of Theorem~\ref{th:new_general_MB_fraction_new} holds if and only if the polarities have non-absolute points of maximum degree. The diameter condition of Theorem~\ref{th:new_general_MB_fraction_new} holds for graphs of equal diameter satisfying Corollary~\ref{cor:ulb}, but may hold for other families as well.

In the following theorem and throughout the paper, $\overline{G}$ denotes the graph $G$ with self-loops removed, giving a consistent basis for graph-order comparison with graphs in the literature. 

\begin{theorem}\label{th:new_general_MB_fraction_new}
    Let $\{\mathbb{H}_q\}$ and $\{\mathbb{K}_r\}$ be infinite families of graphs that satisfy the hypotheses of Theorem \ref{th:main_isomorphism}, and let $\mathcal{P}$ be the infinite family of graphs $\{\overline{P_H(\mathbb{H}_q) \otimes P_K(\mathbb{K}_r)}\}$.  Let $\{\mathbb{H}_q\}$ and $\{\mathbb{K}_r\}$ be such that for all admissible $q$ and $r$,
    \begin{itemize}[leftmargin=*]
        \item $\Delta(\mathbb{H}_q)= \Delta(\overline{P_H(\mathbb{H}_q)}) = \Delta_{H,q},$ $\Delta(\mathbb{K}_r)= \Delta(\overline{P_K(\mathbb{K}_r}))=\Delta_{K,r},$ and   
        \item $D(P_H(\mathbb{H}_q)) = D(P_K(\mathbb{K}_r)) = D(P_H(\mathbb{H}_q) \otimes P_K(\mathbb{K}_r))=D.$ 
    \end{itemize}
    Then

    \begin{enumerate}
    \item The graphs in $\mathcal{P}$ have degree $\Delta_{H,q} \cdot \Delta_{K,r}$ and diameter $D$.
    \item If the $\overline{P_K(\mathbb{K}_r)}$ approach the fraction $c_K$ of the Moore bound as $\Delta_{K,r} \rightarrow \infty,$  
    then the graphs in $\mathcal{P}$ approach the fraction
     $f_q$ of the Moore bound for fixed admissible $q$ 
    as $\Delta_{K,r} \rightarrow \infty$, where    $$\label{eq:limits_new2}
    f_q = c_K \cdot \frac{|P_H(\mathbb{H}_q)|}{(\Delta_{H,q})^D}.
    $$
    \item If the $\overline{P_K(\mathbb{K}_r)}$ approach the fraction $c_K$ of the Moore bound as $\Delta_{K,r} \rightarrow \infty$ and the $\overline{P_H(\mathbb{H}_q)}$ approach the fraction $c_H$ of the Moore bound as $\Delta_{H,q} \rightarrow \infty,$ then the graphs in $\mathcal{P}$ approach the fraction $f$ of the Moore bound 
    as $\Delta_{H,q},\Delta_{K,r} \rightarrow \infty$, where
    $$f=c_K \cdot c_H.$$
    \end{enumerate}
\end{theorem}

\subsection{Paper Organization}
We give basic notation and conventions in Section~\ref{sec:notation}, preliminaries in Section~\ref{sec:preliminaries}, definitions and previously established theorems needed for our results in Section~\ref{sec:background}, proofs of the structural theorems in Section~\ref{sec:main_th}, applications of these theorems to bipartite compatibility graphs in Section~\ref{sec:bipartite_compatibility}, applications of these theorems to generalized polygons in Section~\ref{sec:genpolys_complete}, with a construction of a new largest graph for its degree and diameter.

\section{Notation and Conventions}\label{sec:notation}
\subsection{Notation}
Throughout this paper, we use the notation in Table~\ref{table:notation}.
\begin{table}[!htbp]
\centering
\caption{Notation table. }

\label{table:notation}

\renewcommand{\arraystretch}{1.25}
\begin{tabular}{ll}
\toprule
\textbf{Notation}  &\textbf{Concept}  \\ \midrule
$D(G), \Delta(G)$ & the diameter and maximum degree, respectively, of a graph $G$ \\
$\overline{G}$ &  the graph obtained by dropping any self-loops from a graph $G$\\
$\mathbb{G}$ & ``blackboard bold'' denotes a bipartite graph \\
$H \otimes K$ & the Kronecker product of two connected graphs $H$ and $K$ \\
$\mathbb{H} \odot \mathbb{K}$ & one of the isomorphic components of $\mathbb{H} \otimes \mathbb{K}, $ when $\mathbb{H}, \mathbb{K}$ are bipartite 
\\
$P(\mathbb{G})$ & the polarity quotient graph of a bipartite graph $\mathbb{G}$ admitting polarity \\
$\mathbb{C}(k,q)$ & the bipartite compatibility graph over $\mathbb{F}_q^{2k+1}$ \\
$\mathbb{G}_n(s,t)$ & the generalized $n$-gon of order $(s,t)$ \\
$\mathcal{G}_{n,q}$ &the family of graphs $\overline{P(\mathbb{G}_n(q,q) \odot \mathbb{G}_n(r,r))},$ with $n$ and $q$ fixed \\
$\mathcal{G}_{n}$ &the family of graphs $\overline{P(\mathbb{G}_n(q,q) \odot \mathbb{G}_n(r,r))},$ with $n$ fixed\\
\bottomrule
\end{tabular} 

\end{table}

\subsection{Kronecker Product}
The Kronecker product is a common graph construction that is referred to in other papers as the \emph{direct} \cite{hammack2011handbook, congruences_loop_1, congruences_loop_2}, \emph{categorical} or \emph{tensor} \cite{hammack2011handbook} product of graphs. In \cite{delorme_opp_85}, the term \emph{graph conjunction} is used, and the term ``Kronecker product'' instead denotes one of the two components of the full product.

We assume that for the Kronecker product $H \otimes K$, both $H$ and $K$ are connected.

In some papers, including the original Weichsel paper \cite{kronecker_1962}, the Kronecker product applies only to simple graphs with no self-loops. We work in the category of graphs allowing self-loops but no multiple edges, and using the definition of the Kronecker product from \cite[Section 5.3]{hammack2011handbook}. Disallowing self-loops greatly restricts the discussion of graph homomorphisms, and in particular, the polarity quotients fundamental to this paper.
These self-loops are used in the polarity constructions discussed here, giving extra edges in the Kronecker product and enabling the theorems in the paper.

\subsection{Self-loops}\label{sec:self_loop_convention}
As mentioned above, self-loops that occur in factor graphs are retained for the Kronecker product.
On the other hand, the degree-diameter comparisons in the application portion of the paper make no use of self-loops, as
the presence of self-loops has no effect on the diameter and order of the graph and only adds $2$ to the degree of each vertex with a self-loop. Thus, after constructions are completely finished and we are comparing to the Moore bound, we delete self-loops.  We then use the resulting simple graph in the comparison to compare with simple graphs from the degree-diameter literature, giving a consistent basis for comparison. 

Note that the overline notation $\overline{G}$ signifies that self-loops are deleted \emph{only after the entire construction under the overline is complete}.

\subsection{Generalized Polygons}
When we refer to the graph of a generalized $n$-gon of order $(s,t)$, we will by default use its bipartite point-line incidence-graph representation. 

\subsection{Graph Order vs. Generalized-Polygon Order} The term ``order'' is overloaded in the concepts in this paper: the order of a graph is the number of vertices in the graph, whereas the order of a generalized polygon is a pair of parameters $(s,t)$ defining it. We use both definitions of order in this paper, and where needed for clarity, specify either graph or generalized-polygon order.

\section{Preliminaries}\label{sec:preliminaries}
\subsection{The Kronecker Product and Graph Quotients}\label{sec:prelim_product_quotients}
Several papers have appeared addressing quotients on the Kronecker product.

Delorme~\cite{delorme_opp_85} devised a large-graph construction using the polarity quotient of a generalized polygon and its opposite. This is discussed in more detail in Section~\ref{sec:prior_kronecker_constructions}.

Jha et al.\ discussed graphs having a bipartition-reversing automorphism, and showed that if one of the graphs in a Kronecker product has this property, the product has two isomorphic components \cite{iso_comp_kronecker_bipartite_1997}. Hammack then showed the converse \cite{iso_comp_kronecker_bipartite_2006}. We explore this in detail for the special case of polarity, which is important in the field of finite geometry.

Hammack also studied product cancellation of factors in Kronecker products, and showed conditions on $A$ and $B$ for which $A\otimes C \cong B\otimes C$ implies $A \cong B$ for any $C$ \cite{bipartition_reversing_involutions_08}. Abay-Asmerom et al.\ went on to discuss Kronecker product factorization in the same context \cite{bipartition_reversing_involutions_10}. We have obtained an alternative proof of Theorem~\ref{th:main_isomorphism} using the factorization in \cite[Proposition 1] {bipartition_reversing_involutions_10} as a starting point. However, in so doing, it is still necessary to explicitly track polarities through the factorizations of both factor graphs and their Kronecker product. The alternative proof starting with \cite{bipartition_reversing_involutions_10} thus follows the same path as the direct proof presented here, but with the actions of the polarities encoded in the factorizations, rather than explicitly described as in our proof.

The interaction of graph products and quotients was studied in a series of papers on abstract graph congruences. Broere, Heidema and Pretorius \cite{congruences_no_loop}, Veldsman \cite{congruences_loop_1}, and Broere, Heidema and Veldsman \cite{congruences_loop_2} gave graph-theoretic versions of the Isomorphism Theorems of Algebra (see e.g., \cite{gratzer_universal_alg}), and as a consequence, derive a graph-theoretic version of Birkhoff's Subdirect Decomposition Theorem \cite{Birkhoff1944}: 

\begin{proposition}\cite[Theorem 2.5]{congruences_loop_1} \cite[Theorem 2.8]{congruences_loop_2}\label{prop:congruence_paper}
    Let $I$ be an index set, and for each $i \in I,$ let $\theta_i$ be a congruence on a graph $G$ with $\theta:=\bigcap\limits_{i\in I} \theta_i.$ Then $G/\theta$ is
isomorphic to a subdirect product of the $G/\theta_i,$ $i \in I.$
\end{proposition}
This proposition shows that the quotient is isomorphic to some unspecified subdirect product of corresponding quotient graphs. In contrast, our Theorem~\ref{th:main_isomorphism} shows that the quotient is the full direct product, so is more specific. Its corollary, Corollary~\ref{cor:main_isomorphism}, is also outside the scope of Proposition~\ref{prop:congruence_paper}; $P(\mathbb{H} \otimes \mathbb{K})$ gives two disjoint copies of the full direct product, so cannot be a subdirect product of these two quotient graphs.

Our theorems have an abstract flavor, like Proposition~\ref{prop:congruence_paper}, but are entirely graph-theoretic, addressing bipartite graphs and polarity.

\subsection{The Degree-Diameter Problem and the Moore Bound}\label{sec:moore_bound}
Given a degree $\Delta$ and diameter $D$, the \emph{degree-diameter problem} of graph theory consists of finding a simple graph $G(\Delta,D)$ having largest order possible. The Moore bound~\cite{hoffmansingleton1960} is a theoretical upper bound for the order of such a graph:
\begin{equation}\label{eq:MB}
    \textrm{MB}(\Delta,D) = 1 + \Delta\sum_{i=0}^{D-1} (\Delta-1)^i,    
\end{equation}
so when $D$ is fixed,
\begin{equation}\label{eq:MB_limit}
    \lim_{\Delta \rightarrow \infty} \frac{\textrm{MB}(\Delta,D)}{\Delta^D} = 1,
\end{equation}
with this immediate consequence:
\begin{proposition} \label{prop:limits_MB}
    Let $\{G_i\}$ be a family of graphs of fixed diameter $D$ and degree $\Delta_i \rightarrow \infty.$ Then 
    $$
        \lim_{\Delta_i \rightarrow \infty} \frac{|G_i|}{\text{MB}(\Delta_i, D)} = c \text{ if and only if } \lim_{\Delta_i \rightarrow \infty} \frac{|G_i|}{(\Delta_i)^D} = c.
    $$
\end{proposition}

Simple graphs that reach the Moore bound are called \emph{Moore graphs}. 
Few Moore graphs exist: these include the diameter-$1$ complete graphs, and the degree-$2$ odd cycles. There are also three diameter-$2$ graphs at degrees $2,3$ and $7$, and a hypothetical diameter-$2$ graph of degree $57$, completing the list of diameter-$2$ Moore graphs \cite{hoffmansingleton1960}.
Damerell \cite{Damerell1973OnMG} and Bannai and Ito \cite{Bannai1973OnFM} independently proved that no Moore graphs exist with diameter $D \ge 3$ and degree $\Delta\ge 3$. The above list of Moore graphs is thus complete.

Many researchers have worked to improve the Moore bound; however, for all but a few special families of graphs, MB$(\Delta, D)-2$ is the best currently known upper bound; see the Miller-\v{S}ir\'{a}\v{n} survey \cite{miller2012moore} and the references within. 

\subsection{Largest Known Graphs of a Given Degree and Diameter}\label{sec:largest_known} 
Since the Moore bound has proven difficult to reduce, there have been extensive efforts in constructing largest known graphs for given $\Delta$ and $D$. This establishes lower bounds on the order of the largest possible graphs, and narrows the gap between the largest achieved graphs and the upper bound. Comellas provides a regularly maintained table of largest known graphs of diameters $\le 10$ and degrees between $3$ and $16$ at \cite{deg_diam_table_2026, Comellas2026Dataset}, as a companion to his published \cite{deg_diam_pub}. This table is undergoing rapid change at the time of writing this paper, but in general, graphs in the table are well below $50\%$ of the Moore bound, with larger diameters usually below $10\%$.

\subsection{Known Families That Asymptotically Approach the Moore Bound}
\subsubsection{Generalized-Polygon Constructions Based on Polarity}\label{sec:prior_polarity_constructions}

Several of the largest-known graphs are based on polarities of objects in finite projective spaces. 

Brown \cite{brown_1966} and Erd\"os and R\'enyi \cite{erdosrenyi1962} independently discovered a family of diameter-2 graphs that asymptotically achieve the Moore Bound. These are constructed by applying a polarity to the bipartite incidence graph of a finite projective plane, so these graphs exist for all degrees $q+1$ with $q$ a prime power.

Delorme noted in \cite{delorme_french_polys_replication_85} that polarity quotients of incidence graphs of certain generalized polygons asymptotically approach the Moore bound.  
Using generalized quadrangles, he constructed a family of diameter-3 graphs that exist for degree $\Delta = 2^{2m+1} + 1$ with $m \in \mathbb{Z}_{\ge 0}$.  These graphs are based on ovoids in $\mathbb{P}^3(\F_q)$, see e.g. \cite[Section 7.3]{GenPolys_VanMaldgehem} for more details. For diameter 5, he constructed a similar family of graphs using generalized hexagons. This family is also sparse, and these graphs only exist for degrees $\Delta=3^{2m+1}+1,$ $m \in \mathbb{Z}_{\ge 0}$. 

In \cite{siagiova_siran}, \v{S}iagiov{\'a} and \v{S}ir{\'a}\v{n} constructed  Cayley graphs of diameter $2$ and order $\Delta^2-\mathcal{O}(\Delta^{1.5})$ that approach the Moore bound. In \cite{bachraty_siran_2015}, Bachrat\'y and \v{S}ir{\'a}\v{n} observed that these graphs could be constructed using polarity quotients of generalized triangles with polarity. Similarly, in \cite{bachraty_siagiova_siran},
Bachrat\'y, \v{S}iagiov{\'a} and \v{S}ir{\'a}\v{n} constructed Cayley graphs of diameter $3$ and order $\Delta^3-\mathcal{O}(\Delta^{2.5})$  that approach the Moore bound, by modifying subgraphs of the polarity quotients of generalized quadrangles with polarity. Their method does not extend to generalized hexagons.

\subsubsection{Generalized-Polygon Constructions Based on Kronecker Products}\label{sec:prior_kronecker_constructions}

Delorme constructed another family of diameter-$3$ graphs approaching the Moore bound in \cite{delorme_opp_85}, based on the polarity quotient (with any self-loops removed) of the Reduced-Kronecker product of a large bipartite graph $\mathbb{G}(\mathbb{U}=(\mathbb{A},\mathbb{B}),\mathbb{E})$ and its opposite $Opp(\mathbb{G}) = \mathbb{G}(\mathbb{U}=(\mathbb{B},\mathbb{A}),\mathbb{E})$, where the sets $\mathbb{A}$ and $\mathbb{B}$ are interchanged but adjacency is preserved. 
    Large graphs are constructed when the base graph $\mathbb{G}$ is large: for example, Delorme used the incidence graphs of the generalized quadrangles $\mathbb{G}_4(s, t)$ to construct graphs with order $(1+s)(1+t)(1+st)^2$ and degree $(1+s)(1+t)$. (Generalized quadrangles exist for $(s=q,t=q), (s=q-1, t=q+1), (s=q, t=q^2)$ and $(s=q^2, t=q^3)$ for prime powers $q$.)
    In \cite{delorme_opp_85}, the second factor $Opp(\mathbb{G}_n(s,t))$ is determined by the first factor $\mathbb{G}_n(s,t)$, unlike our construction $P_m(\mathbb{G}_m) \otimes P_n(\mathbb{G}_n)$, where the two factors may be any pair of generalized polygons that both admit polarity. Conversely, the factors in  \cite{delorme_opp_85} need not have a polarity. These two constructions thus produce different sets of graphs. 

     \subsubsection{A New Construction Associated with High-Rank Spherical Buildings} Cames van Batenburg and Korsky recently introduced a family of graphs for any diameter $D$ that asymptotically approaches the Moore bound as the degree $\Delta \rightarrow \infty$ \cite{vanbatenburg2026asymptoticallyattainingmoorebound}. This proves a longstanding conjecture of Bollob\'as \cite{Bollobas1978, CambieCamesVanBatenburgDeJoannisDeVerclosKang2022}: for every fixed $D \ge 2$ and every $\epsilon > 0$, there are infinitely many $\Delta$ for which $n_D(\Delta) \ge (1-\epsilon)\Delta^D$, where $n_D(\Delta)$ is the largest order of a graph having diameter at most $D$ and maximum degree at most $\Delta$. 
     
     This construction has as a starting point the flag geometry of the type-$A_{2D}$ spherical building associated with $V=\mathbb{F}_q^{2D+1}$. The construction proceeds by considering the set of complete flags in $V$, and splitting these flags into even and odd partial flags $\mathcal{P}_{even}(V)$ and $\mathcal{P}_{odd}(V)$. An even-rank and an odd-rank partial flag are considered \emph{compatible} when their members, arranged in ascending order by rank, form a complete flag. The set $\mathcal{P}_{even}(V)$ is the set of vertices of a graph $H_{D,q}$, where two even partial-flag vertices are adjacent when they share a common compatible odd partial-flag. For every fixed diameter $D$, the graphs $H_{D,q}$ have diameter $D$, and the family $\{H_{D,q} \mid q \text{ a prime power}\}$ approaches the Moore bound as $q \rightarrow \infty.$ In fact, the authors prove a stronger variant of the Bollob\'as conjecture: for every fixed integer $D$, $\lim_{\Delta \rightarrow \infty}\frac{n_D(\Delta)}{\Delta^D}=1$.

     We present in this paper a different construction, starting with the bipartite compatibility graphs that were used in \cite{vanbatenburg2026asymptoticallyattainingmoorebound}, but then taking another line of development. We produce two families of graphs of even diameter $2k$, with better degree coverage.

\section{Background: Products and Polarities}\label{sec:background}
In this section, we present a number of definitions and propositions that are already known or that follow immediately from known propositions, along with some discussion. These are used later in the proofs of the theorems of this paper.
\subsection{The Kronecker Product}\label{sec:kronecker}
\begin{definition} [Kronecker Product of Graphs] \cite{kronecker_1962}\label{def:kronecker} 
    Given two connected graphs $H(U_H, E_H)$ and $K(U_K, E_K)$ that permit self-loops but have no multi-edges, their Kronecker product is the graph $G(U_H\otimes U_K,E)$, where $((v_H,v_K), (w_H,w_K)) \in E$ if and only if $(v_H, w_H) \in E_H$ and $(v_K, w_K) \in E_K$. We denote this product by $H \otimes K$.
\end{definition}

\begin{proposition} \cite{hammack2011handbook} \label{prop:kronecker_props}
    Let $H$ and $K$ be graphs. Then $|V(H\otimes K)|= |V(H)||V(K)|.$
\end{proposition}
\begin{proposition} \label{prop:diam_kron}  \cite[Proposition 5.7]{hammack2011handbook}
    Let $H$ and $K$ be graphs of diameters $D(H)$ and $D(K)$. Then 
    $\max(D(H), D(K)) \le D(H \otimes K).$ 
    \end{proposition}

Propositions \ref{prop:kronecker_bipartite} and \ref{prop:vertices_kronecker} come from papers that define Kronecker only for loopless graphs. However, these propositions assume bipartite graphs, which must be loopless, so their Kronecker product is the same under either definition. Thus, the propositions apply in either case, and in particular, under the looped-graph Kronecker definition used here.
\begin{proposition}\label{prop:kronecker_bipartite}\cite[Corollary of Theorem 1]{kronecker_1962}
Let $\mathbb{H}$ and 
$\mathbb{K}$ be two bipartite graphs. Then their Kronecker product $\mathbb{G} = \mathbb{H} \otimes \mathbb{K}$ is composed of two disconnected bipartite graphs.
\end{proposition}
\begin{proposition}\label{prop:vertices_kronecker}~\cite[Lemma 2.2]{iso_comp_kronecker_bipartite_1997}
    Let $\mathbb{H}(U_H=(A_H,B_H),\mathbb{E}_H)$ and 
$\mathbb{K}(U_K=(A_K,B_K),\mathbb{E}_K)$ be bipartite graphs. Then the sets of vertices of the two components of $\mathbb{H} \otimes \mathbb{K}$ are $(A_H \times A_K) \cup (B_H \times B_K)$ and $(A_H \times B_K) \cup (B_H \times A_K)$.
\end{proposition}

\subsection{Polarities} \label{sec:polarities}
\begin{definition}[Polarity] \label{defn:polarity_alt3} 
A graph automorphism $\rho$ on a bipartite graph $\mathbb{G}(U=(A,B),\mathbb{E})$ is a polarity if $\rho:A \rightarrow B$ is an involution.
\end{definition}

The concept of polarity originates in the context of incidence graphs of finite geometries (see \cite{delorme_opp_85, GenPolys_VanMaldgehem} and others) and is widely discussed in that field. This was extended to the class of bipartite graphs by  Lazebnik and Woldar in \cite{lw_polarity_graphs}. Hammack \cite{bipartition_reversing_involutions_08} and Abay-Asmerom et al.\ \cite{bipartition_reversing_involutions_10} use the terminology \emph{bipartition-reversing involution} for the same concept, or just \emph{reversing involution} when the bipartite context is clear.
A polarity always reverses the partition of a bipartite graph, so is a special case of an involutive automorphism.
\begin{definition} 
[Polarity Quotient Graph]\label{def:pol_quotient_graph}
    Let $\mathbb{G}$ be a bipartite graph $\mathbb{G}(U=(A,B),\mathbb{E})$ and let $\rho$ be a polarity. The polarity quotient graph is the quotient graph $P(\mathbb{G})$, where vertices are the orbits of the quotient $[v]=\{v, \rho(v)\}$ and $[v]$ and $[w]$ are adjacent whenever there exist $v'\in [v]$ and $w'\in [w]$ such that $(v',w') \in E.$ Multi-edges are merged into a single edge but self-loops are retained. 
\end{definition}

The notion of a polarity graph with respect to generalized polygons is well known and is discussed in \cite{delorme_opp_85} and elsewhere. However, in that context, self-loops are usually deleted. The \emph{quotient graph}, with self-loops retained, is defined in \cite{congruences_loop_1, congruences_loop_2}, and Definition~\ref{def:pol_quotient_graph} gives the quotient graph as defined in those papers, with respect to the polarity.
\begin{definition}[Absolute Point] \label{def:abs_point}
    An absolute point with respect to a polarity $\rho$ of a bipartite graph $\mathbb{G}$ is a vertex $v \in V(\mathbb{G})$ such that $(v, \rho(v)) \in E(\mathbb{G}).$ 
\end{definition}
Definition~\ref{def:abs_point} extends the terminology of finite geometry found in \cite{GenPolys_VanMaldgehem} and elsewhere to the general case of bipartite graphs. This definition, for bipartite graphs, is also seen in \cite{lw_polarity_graphs}.
\begin{definition}[Absolute Vertex of a Polarity Quotient]\label{def:abs_vertices}
    An absolute vertex of a polarity quotient of a bipartite graph $\mathbb{G}$ is a vertex of the polarity quotient graph $P(\mathbb{G})$ resulting from the contraction of $(v, \rho(v)) \in \mathbb{G}$ corresponding to the orbit of $\{v, \rho(v)\}$ in the quotient, where $v$ is an absolute point. Consequently, the polarity induces a self-loop at the corresponding vertex in the quotient graph $P(\mathbb{G})$.
\end{definition}

\begin{proposition} \label{prop:polarity_props} \cite[Section 2]{delorme_opp_85}
    Let $\mathbb{G}$ be a bipartite graph $(U=(A,B),\mathbb{E})$ with a polarity map $\rho:A\rightarrow B$, and let $\overline{P(\mathbb{G})}$ be the polarity quotient graph of $\mathbb{G}$ with any resulting self-loop edges dropped.
Then
\begin{enumerate}
    \item $\Delta(\overline{P(\mathbb{G})}) \le \Delta(\mathbb{G})$,   
    \item $D(\overline{P(\mathbb{G})}) \le D(\mathbb{G})-1$,  \text{and}
    \item $|V(\overline{P(\mathbb{G}))}| = |A| = |B| =\frac{|V(\mathbb{G})|}{2}.$
\end{enumerate}
\end{proposition}
The following lemma is stated in \cite[Theorem 2.7(i)]{lozpineda2010} for the special case of generalized polygons. We generalize it to any bipartite graph with a polarity.
\begin{lemma} \label{lemma:polarity_props_diam_equality} 
    Let $\mathbb{G}$ be a bipartite graph with polarity $\rho$, and let $\overline{P(\mathbb{G})}$ be the polarity quotient graph of $\mathbb{G}$ with self-loop edges dropped. Then $\Delta(\overline{P(\mathbb{G})}) = \Delta(\mathbb{G})$ if and only if there exist points in $\mathbb{G}$ of maximal degree that are not absolute with respect to $\rho.$
\end{lemma}
\begin{proof}
    Let $v$ be a vertex of $G$. Since $\rho$ is an automorphism, each edge $(v,w)$ induces an edge
$([v], [w])$ in $P(G)$. These edges are distinct: since $G$ is bipartite and $\rho$ is a polarity, no more
than one of $w$ and $\rho(w)$ can be a neighbor of $v$. Thus, the number of edges incident to $[v]$ is
equal to the degree of $v$. If $v$ is absolute with respect to $\rho$, then $[v]$ has a self-loop deleted
in $P(G)$, i.e., one fewer edge incident to $[v]$. All remaining edges join $[v]$ and a neighbor in
$P(G)$, so in this case, $[v]$ will have degree deg$(v)-1$ in $P(G)$. So vertices $[v]$ from absolute
points $v$ have degree $1$ smaller than $v$, completing the proof.
\end{proof}
\subsection{Bipartition-Reversing Automorphisms and Kronecker Products}\label{sec:pautomorphisms_kronecker}
An automorphism on a bipartite graph must either reverse or preserve its partitions \cite{bipartition_reversing_involutions_10}.
In \cite{iso_comp_kronecker_bipartite_1997}, Jha et al. gave several results on Kronecker products of bipartite graphs that have a reversing automorphism, i.e., that exchanges the partitions. We will use their results to show theorems on Kronecker products of bipartite graphs that admit polarity, enabling the main theoretical results of this paper. 

We recall from Section~\ref{sec:self_loop_convention} that any self-loops that emerge during construction are retained until we conduct degree/diameter comparisons to other graphs at the end of the construction, at which point we drop the self-loops.

\begin{definition}[Bipartition-Reversing Automorphism]~\cite[Definition 3.1]{iso_comp_kronecker_bipartite_1997} \cite[Page 2043]{bipartition_reversing_involutions_10}\label{prop:pi}
    A bipartite graph $\mathbb{G} = (U=(A,B), \mathbb{E})$ has a bipartition-reversing automorphism when $\mathbb{G}$ admits an automorphism $\psi: A \rightarrow B$. 
    \emph{(}In \cite{iso_comp_kronecker_bipartite_1997}, a graph having such an automorphism is said to have Property $\pi$.\emph{)}
\end{definition}

A bipartition-reversing automorphism is called a  \emph{duality} when discussing incidence graphs of finite geometries \cite{delorme_french_polys_replication_85, delorme_opp_85, GenPolys_VanMaldgehem}. The theorems of Section~\ref{sec:main_bipartite} apply to bipartite graphs, but one of our applications is from finite geometry, so we mention duality here for context. 

The forward implication of Theorem~\ref{th:iso_components_kronecker} was shown by Jha et al.\ in 1997 \cite{iso_comp_kronecker_bipartite_1997}. The converse remained a conjecture until Hammack proved it in 2009 \cite{iso_comp_kronecker_bipartite_2006}. 
\begin{theorem}\label{th:iso_components_kronecker} ~\cite[Theorem 3.2]{iso_comp_kronecker_bipartite_1997} \cite[Theorem 1]{ iso_comp_kronecker_bipartite_2006} 
    Let $\mathbb{H}$ and $\mathbb{K}$ be bipartite graphs. The two components of $\mathbb{H} \otimes \mathbb{K}$ are isomorphic if and only if at least one of $\mathbb{H}$ and $\mathbb{K}$ has a bipartition-reversing automorphism. 
\end{theorem}

\begin{definition} [Reduced-Kronecker product of graphs] \cite{iso_comp_kronecker_bipartite_1997,iso_comp_kronecker_bipartite_2006} \label{def:reduced_kronecker}
    Let  $\mathbb{H}$ and $\mathbb{K}$ be bipartite graphs, at least one of which has a bipartition-reversing automorphism. Their Reduced-Kronecker product $\mathbb{H} \odot \mathbb{K}$ is one of the two components of their Kronecker product $\mathbb{H} \otimes \mathbb{K}$, which by Theorem~\ref{th:iso_components_kronecker} are isomorphic.
\end{definition}

Proposition~\ref{prop:redkronecker_stats}$(2)$ is due to Pu\v{s} \cite[Corollary 1]{diam_kronecker}, who attributes it to Hell \cite{diam_kronecker_orig}. Propositions~\ref{prop:redkronecker_stats}$(1)$ and $(3)$ follow directly from Theorem~\ref{th:iso_components_kronecker} and Proposition~\ref{prop:vertices_kronecker} of \cite{iso_comp_kronecker_bipartite_1997}, and  the definition of Kronecker product.
\begin{proposition} \label{prop:redkronecker_stats}
Let $\mathbb{H}(U_H=(A_H,B_H),\mathbb{E}_H)$ and 
$\mathbb{K}(U_K=(A_K,B_K),\mathbb{E}_K)$ be connected bipartite graphs, at least one of which has a bipartition-reversing automorphism. Then the degree $\Delta$, diameter $D$  and number of vertices $|V(\mathbb{G})|$ of the Reduced-Kronecker product $\mathbb{G}=\mathbb{H} \odot \mathbb{K}$ are
\begin{enumerate}
\item $\Delta(\mathbb{G})=\Delta(\mathbb{H})\Delta(\mathbb{K}),$
\item  $D(\mathbb{G})=\max(D(\mathbb{H}), D(\mathbb{K})),$ and
\item $|V(\mathbb{G})|=|A_H||A_K|+|B_H||B_K|.$ \end{enumerate}
\end{proposition}

\begin{proposition}\cite[Proposition 3.3]{iso_comp_kronecker_bipartite_1997}\label{prop:product_pi}
    Let $\mathbb{H}$ and $\mathbb{K}$ be bipartite graphs having bipartition-reversing automorphisms $\rho_H$ and $\rho_K$. Then the isomorphic components of $\mathbb{H} \otimes \mathbb{K}$ also have a bipartition-reversing automorphism $\rho((h, k)) = (\rho_H(h), \rho_K(k))$. 
\end{proposition}
\subsection{Polarities and Kronecker Products}\label{sec:polarities_kronecker} 

A polarity is a bipartition-reversing automorphism, so any theorem regarding graphs with bipartition-reversing automorphisms also holds for graphs admitting polarity. 
 Corollary~\ref{cor:product_polarity} then follows immediately from Proposition~\ref{prop:product_pi}, since the automorphism on $\mathbb{H} \odot \mathbb{K}$ will be an involution, as will that on $\mathbb{H} \otimes \mathbb{K}$.
\begin{corollary}\label{cor:product_polarity} 
Let $\mathbb{H}$ and $\mathbb{K}$ be bipartite graphs that admit polarity. Then both $\mathbb{H} \otimes \mathbb{K}$ and $\mathbb{H} \odot \mathbb{K}$ admit polarity, with the polarity $\rho$ as in Proposition~\ref{prop:product_pi}.
\end{corollary}

\begin{definition}[Composed Polarity]\label{def:composed_polarity}
    Let $\mathbb{H}$ and $\mathbb{K}$ be bipartite graphs with polarities $\rho_H$ and $\rho_K$. The \emph{composed polarity} $\rho$ on  $\mathbb{H} \otimes \mathbb{K}$ and on $\mathbb{H} \odot \mathbb{K}$ is the map $\rho((h, k)) = (\rho_H(h), \rho_K(k)).$ 
\end{definition}

We observe that when $\mathbb{H}$ and $\mathbb{K}$ are bipartite graphs admitting polarity, $\mathbb{H} \otimes \mathbb{K}$ admits other polarities beyond the composed polarity $\rho$.  Let $\gamma$ be an isomorphism from one component of $\mathbb{H} \otimes \mathbb{K}$ to the other. 
Any automorphism $\psi$ of $\mathbb{H} \odot \mathbb{K}$ induces an isomorphism $\phi_{\gamma,\psi} \ne \rho: \phi_{\gamma,\psi}(v) = \psi(\gamma(v))$, sending one component of $\mathbb{H} \otimes \mathbb{K}$ onto the other, using $\psi$ in one direction and $\psi^{-1}$ in the other. It is therefore a polarity of $\mathbb{H} \otimes \mathbb{K}$ arising from one of its two bipartitions. 
The polarity quotient of $\phi_{\gamma,\psi}$ is isomorphic to $\mathbb{H} \odot \mathbb{K},$ so differs from the polarity quotient of $\rho$ (discussed in Corollary~\ref{cor:main_isomorphism} of Theorem~\ref{th:main_isomorphism}). 

Different polarities on isomorphic graphs may yield nonisomorphic polarity quotients with the same maximum degree, diameter and order; examples of this are given in Table~\ref{table:design_space_new}. We limit our discussion to the composed polarity $\rho$ for the rest of this paper, since this polarity leads to the structural and degree-diameter results in this paper. 

\section{Proofs of Structural Results}\label{sec:main_bipartite} 
\label{sec:main_th}

Here we prove Theorems \ref{th:main_isomorphism} and \ref{th:ulb} given in Section \ref{sec:intro_contributions}.
The theorems and corollaries from this section generalize easily to multiple factor graphs.

\subsection{Proofs of Structural Theorems and Corollaries}
\begin{proof}[
  Proof of Theorem~\ref{th:main_isomorphism}
]
	By Corollary~\ref{cor:product_polarity}, the polarity $\rho$ restricts to a polarity on $\mathbb{H} \odot \mathbb{K}.$ 
    We define a map 
	\begin{align*}
		\phi: P(\mathbb{H} \odot \mathbb{K}) &\rightarrow P_H(\mathbb{H}) \otimes P_K(\mathbb{K}),\\
		[(h,k)] &\rightarrow  ([h]_H, [k]_K),
	\end{align*}
	where  $[(h,k)]$ is the equivalence class $\{(h,k), \rho((h,k))\} $, $[h]_H$ is the equivalence class $\{h, \rho_H(h)\}$, and $[k]_K$ is the equivalence class $\{k, \rho_K(k)\},$ under their respective polarities. 
	
	The two components of $\mathbb{H} \otimes \mathbb{K}$ are isomorphic, so without loss of generality, we let the vertices of $\mathbb{H} \odot \mathbb{K}$ be $(A_H \times A_K) \cup (B_H \times B_K)$ as in Proposition~\ref{prop:vertices_kronecker}, and recall that $\rho_H$ maps $A_H$ to $B_H$ and $\rho_K$ maps $A_K$ to $B_K$. By the definition of composed polarity, $\{(h,k), \rho((h,k))\}= \{(h,k), (\rho_H(h),\rho_K(k))\},$ 
	so 
    \begin{align}
		V(P(\mathbb{H} \odot \mathbb{K})) 
		&=\{ \{(h,k), (\rho_H(h),\rho_K(k))\} \mid h \in A_H, k \in A_K\}.   \label{eq:P_h_dot_k_reps}     \intertext{and by the definition of the Kronecker product,}
		V(P_H(\mathbb{H}) \otimes P_K(\mathbb{K})) &= V(P_H(\mathbb{H})) \times V(P_K(\mathbb{K}))  \nonumber \\ 
		&= \{(\{h,\rho_H(h)\},\{k,\rho_K(k)\}) \mid h \in A_H, k \in A_K \}. \nonumber
        \intertext{We then have that}
        \phi: \{(h,k), (\rho_H(h),\rho_K(k))\} &\rightarrow (\{h,\rho_H(h)\},\{k,\rho_K(k)\}), \text{ where } h \in A_H \text{ and } k \in A_K, \nonumber
	\end{align} 
	so $\phi$ is a well-defined bijective vertex mapping from $P(\mathbb{H} \odot \mathbb{K})$ to $P_H(\mathbb{H}) \otimes P_K(\mathbb{K}).$ 
	
	We finish by showing vertex adjacency. Let $[(h_0,k_0)]$ and $[(h_1,k_1)]$ be vertices of $P(\mathbb{H} \odot \mathbb{K})$.  By Equation~(\ref{eq:P_h_dot_k_reps}), 
    \begin{align*}
        [(h_0,k_0)] &= \{(h_0,k_0), (\rho_H(h_0),\rho_K(k_0))\}, \text{ and} \\
        [(h_1,k_1)] &= \{(h_1,k_1), (\rho_H(h_1),\rho_K(k_1))\}.
    \end{align*}
    By Definition~\ref{def:pol_quotient_graph}, $[(h_0,k_0)]$ and $[(h_1,k_1)]$ are adjacent in $P(\mathbb{H} \odot \mathbb{K})$ if and only if there exist $(h_0',k_0')\in [(h_0,k_0)]$ and $(h_1',k_1')\in [(h_1,k_1)]$ such that $(h_0',k_0')$ and $(h_1',k_1')$ are adjacent in $\mathbb{H} \odot \mathbb{K}.$ By the definition of the Kronecker product, this is true if and only if $h_0'$ is adjacent to $h_1'$ in $\mathbb{H}$ and $k_0'$ is adjacent to $k_1'$ in $\mathbb{K}.$ Since the composed polarity $\rho$ satisfies $\rho(h,k) = (\rho_H(h),\rho_K(k))$, by Definition~\ref{def:composed_polarity}, representatives of $[(h_0,k_0)]$ and $[(h_1,k_1)]$ are adjacent 
    if and only if representatives of $[h_0]_H$ and  $[h_1]_H$  are adjacent in $H$ and representatives of $[k_0]_K$ and  $[k_1]_K$ are adjacent in $K$. By Definition~\ref{def:pol_quotient_graph}, this is true if and only if $[h_0]_H$ is adjacent to  $[h_1]_H$ in $P_H(\mathbb{H})$ and $[k_0]_K$ is adjacent to  $[k_1]_K$ in $P_K(\mathbb{K}).$ This is true if and only if $([h_0]_H, [k_0]_K)$ is adjacent to $([h_1]_H, [k_1]_K)$ in $P_H(\mathbb{H}) \otimes P_K(\mathbb{K}),$ again by the definition of the Kronecker product.
    
    Thus, $\phi$ is a graph isomorphism.
\end{proof}
\begin{corollary}
\label{cor:main_isomorphism}
    Let $\mathbb{H}$ and $\mathbb{K}$ be as in Theorem~\ref{th:main_isomorphism}. Then $P(\mathbb{H} \otimes \mathbb{K})$ is isomorphic to $(P_H(\mathbb{H}) \otimes P_K(\mathbb{K})) \sqcup (P_H(\mathbb{H}) \otimes P_K(\mathbb{K})).$
\end{corollary}

\begin{proof}
    Let $\mathbb{G}_1$ and $\mathbb{G}_2$ be the isomorphic components of $\mathbb{H} \otimes \mathbb{K}.$ By Corollary~\ref{cor:product_polarity}, the composed polarity $\rho$ restricts to a polarity on each isomorphic component of  $\mathbb{H} \otimes \mathbb{K}$, so we may apply Theorem~\ref{th:main_isomorphism} to either component of $\mathbb{H} \otimes \mathbb{K}$: $$P(\mathbb{H} \otimes \mathbb{K}) = P(\mathbb{G}_1 \sqcup \mathbb{G}_2) = P(\mathbb{G}_1) \sqcup P(\mathbb{G}_2) \cong (P_H(\mathbb{H}) \otimes P_K(\mathbb{K})) \sqcup (P_H(\mathbb{H}) \otimes P_K(\mathbb{K})).$$ 
\end{proof}

\begin{lemma}\label{lemma:diam_ub}
    Let $\mathbb{H}$ and $\mathbb{K}$ be bipartite graphs admitting polarity with diameters $D(\mathbb{H})$ and $D(\mathbb{K})$ and let $P(\mathbb{H} \odot \mathbb{K})$ be the composed-polarity quotient on $\mathbb{H} \odot \mathbb{K}.$ Then $$D(P(\mathbb{H} \odot \mathbb{K})) \le \max(D(\mathbb{H}),D(\mathbb{K}))-1.$$ 
    \end{lemma}
\begin{proof}
    By Corollary~\ref{cor:product_polarity}, $\mathbb{H} \odot \mathbb{K}$ admits polarity, so has a bipartition-reversing automorphism. So by Proposition~\ref{prop:redkronecker_stats}$(2)$, $ D(\mathbb{H} \odot \mathbb{K})-1 = \max(D(\mathbb{H}),D(\mathbb{K}))-1.$ By Proposition~\ref{prop:kronecker_bipartite} and the definition of the Reduced-Kronecker product, $\mathbb{H} \odot \mathbb{K}$ is bipartite, so by Proposition~\ref{prop:polarity_props}$(2)$, $D(P(\mathbb{H} \odot \mathbb{K})) \le D(\mathbb{H} \odot \mathbb{K})-1 = \max(D(\mathbb{H}),D(\mathbb{K}))-1.$ 
\end{proof}
Theorem \ref{th:ulb} and Corollary~\ref{cor:ulb} follow quickly from the preceding results. 

\begin{proof}[Proof of Theorem~\ref{th:ulb}]
    $P(\mathbb{H} \odot \mathbb{K}) \cong P_H(\mathbb{H}) \otimes P_K(\mathbb{K})$ by Theorem~\ref{th:main_isomorphism}, so this follows from Proposition~\ref{prop:diam_kron} and Lemma \ref{lemma:diam_ub}. 
\end{proof}
\begin{proof}[Proof of Corollary~\ref{cor:ulb}]
    Follows from Theorem~\ref{th:ulb}: $$\max(D(P_H(\mathbb{H})), D(P_K(\mathbb{K})))=\max(D(\mathbb{H}), D(\mathbb{K}))-1.$$
\end{proof}
\begin{proof}[Proof of Theorem~\ref{th:new_general_MB_fraction_new}]
The order of $P_H(\mathbb{H}_q) \otimes P_K(\mathbb{K}_r)$ is $|P_H(\mathbb{H}_q)|\cdot |P_K(\mathbb{K}_r)|,$ by the definition of the Kronecker product. Lemma~\ref{lemma:polarity_props_diam_equality}, Proposition~\ref{prop:redkronecker_stats} and the condition on degrees  give that the degrees of the graphs in $\mathcal{P}$ are $\Delta_{H,q} \cdot \Delta_{K,r}$. By Theorem~\ref{th:main_isomorphism} and the diameter condition, diameter is $D$ for all graphs in these families. 
We expand the Moore bound ratio:
$$
    \frac{|P_H(\mathbb{H}_q)| \cdot |P_K(\mathbb{K}_r)|}{MB(\Delta_{H,q} \cdot \Delta_{K,r},D)} 
    = \frac{|P_K(\mathbb{K}_r)|}{(\Delta_{K,r})^D} \cdot \frac{|P_H(\mathbb{H}_q)|}{(\Delta_{H,q})^D} \cdot \frac{(\Delta_{H,q}\cdot \Delta_{K,r})^D}{MB(\Delta_{H,q} \cdot \Delta_{K,r},D)}
$$
If $q$ and thus $\Delta_{H,q}$ are held fixed, we apply Proposition~\ref{prop:limits_MB} and Equation (\ref{eq:MB_limit}) to get 
\begin{align*}
     f_q
    &=\lim_{\Delta_{K,r} \rightarrow \infty}\left(\frac{|P_K(\mathbb{K}_r)|}{(\Delta_{K,r})^D} \cdot \frac{|P_H(\mathbb{H}_q)|}{(\Delta_{H,q})^D} \cdot \frac{(\Delta_{H,q}\cdot \Delta_{K,r})^D}{MB(\Delta_{H,q} \cdot \Delta_{K,r},D)}\right)\nonumber \\
    &=c_K \cdot \frac{|P_H(\mathbb{H}_q)|}{(\Delta_{H,q})^D}
    \end{align*}
If both $\Delta_{H,q}, \Delta_{K,r} \rightarrow \infty,$  we apply Proposition~\ref{prop:limits_MB} twice and Equation (\ref{eq:MB_limit}) once to get
\begin{align*}
     f
    &=\lim_{\Delta_{H,q},\Delta_{K,r} \rightarrow \infty}\left(\frac{|P_K(\mathbb{K}_r)|}{(\Delta_{K,r})^D} \cdot \frac{|P_H(\mathbb{H}_q)|}{(\Delta_{H,q})^D} \cdot \frac{(\Delta_{H,q}\cdot \Delta_{K,r})^D}{MB(\Delta_{H,q} \cdot \Delta_{K,r},D)}\right) \\
   &= c_K \cdot c_H.
\end{align*}
\end{proof}
\subsection{Applications}
Any pair of graphs that satisfy the conditions of Corollary~\ref{cor:ulb} may be used in the construction here to obtain new graphs whose exact diameter is known. 
Many examples of such graphs may be constructed: for example, the family of ladder graphs $L_n$ with the rail interchange polarity has this property, as does the family of graphs made by attaching two copies of the star graph $K_{1,n}, n\ge 2,$ by an edge joining the center vertices with the pair-branch interchange polarity.

We discuss in the following sections two important examples of graphs of this type: the bipartite compatibility graphs and the generalized $n$-gons that admit polarity.  The polarity quotient graphs of each of these families reach the Moore bound asymptotically, as do Kronecker products of these polarity quotients when the factors have the same fixed diameter and both degrees tend to infinity.

\section{Structural Results Applied to High-Rank Spherical Buildings}\label{sec:bipartite_compatibility}

Here, we apply the results from Section~\ref{sec:main_th} to the special case of flag geometries of high-rank spherical buildings to construct new families of large graphs. 
This construction takes as a starting point the bipartite compatibility graph construction $\mathbb{C}(k,q)$ appearing as the initial idea in \cite{vanbatenburg2026asymptoticallyattainingmoorebound}. The $\mathbb{C}(k,q)$ are constructed using the odd- and even-dimensional partial flags from the spherical building of type $A_{2k}$ associated with $\mathbb{F}_q^{2k+1}$ \cite[Section 4.3]{AbramenkoBrown2008}. In \cite{vanbatenburg2026asymptoticallyattainingmoorebound}, the authors use these to construct infinite families of graphs  of any diameter $k$ and degree $\sim (q+1)^{2k}$ that approach the Moore bound. 

Our construction proceeds along different lines from the halved-flag construction in \cite{vanbatenburg2026asymptoticallyattainingmoorebound}. 
We first establish a natural polarity on the bipartite compatibility graphs. This polarity produces a family of graphs of even diameter $2k$ and of degree $\Delta_q = (q+1)^{k}$, with $q$ a prime power,  that approaches the Moore bound asymptotically as the degrees $\Delta_q \rightarrow \infty.$ 

We also find that these graphs with their natural polarity satisfy the conditions of the theorems and corollaries in Section~\ref{sec:main_th}, so we may construct Kronecker products of these graphs having the same diameter. This gives a second family of graphs of even diameter $2k$  and of degree $\Delta_{q,r} = [(q+1)(r+1)]^{k}$, with $q,r$ prime powers, that approaches the Moore bound asymptotically as the degrees $\Delta_{q}, \Delta_{r} \rightarrow \infty.$ 

We note that the graphs $\mathbb{C}(1,q)$ are isomorphic to the incidence graphs of the Desarguesian generalized triangles $PG(2,q)$, which are rank-$2$ spherical buildings of type $A_2$ \cite[Section 0.5]{AbramenkoBrown2008}. The Erd{\H o}s-R\'enyi-Brown polarity graphs are polarity quotients of the incidence graphs of these buildings \cite{erdosrenyi1962, brown_1966}, thus the constructions in this section generalize the classical Erd{\H o}s-R\'enyi–Brown polarity construction to higher rank. 

We do not consider this case $k=1$ further in this section, but assume that $k>1$ when discussing the $\mathbb{C}(k,q)$ constructions. We discuss these generalized triangles in more detail in Section~\ref{sec:genpolys_complete}, in the context of generalized polygons.

\subsection{Some Preliminaries}
\begin{definition}[Flag, Complete Flag]\cite[page 101 (1.114)]{StanleyEnumComb_2012}
    Let $V$ be a $t$-dimensional vector space. A \emph{flag} in $V$ is an increasing chain of subspaces of $V,$ $F=(W_{1} \subset \dots \subset W_{m}).$ A \emph{complete flag} in $V$ is a flag with a subspace of every dimension from $0$ to $t$. 
\end{definition}
\begin{definition} [Orthogonal Complement of a Flag] \cite[Equations 7.8, 7.9]{Taylor1992}\label{def:orthcomp}
    Let $V=\mathbb{F}_q^{2k+1},$ where $q$ is a prime power and $k \in \mathbb{Z}_{>0},$ and let $F=(W_{1} \subset \dots \subset W_{m})$ be some flag on $V$. Let $B$ be a nondegenerate symmetric bilinear form on $V$. The {orthogonal complement} of the flag $F$ with respect to $B$ is the chain $F^\perp=(W_{m}^\perp, \dots, W_{1}^\perp),$ where $W_i^\perp$ is the set of all vectors of $V$ orthogonal to $W_i$ with respect to $B$.
\end{definition}
\begin{definition}[$q$-Integer $t$, $q$-Factorial of $t$] \cite{StanleyEnumComb_2012} \label{def:q_factorial_t}
    The \emph{$q$-integer $t$} (\emph{$q$ analogue of $t$}) is $$[t]_q = 1+q+ \dots + q^{t-1} = \frac{q^t-1}{q-1}= |PG(t-1,q)|.$$ The \emph{$q$-factorial of $t$}  (\emph{$q$ analogue of $t!$}) is $$[t]_q!=1 (1+q)\dots (1+q+\dots +q^{t-1}) = \prod_{j=1}^{t}[j]_q.$$
\end{definition}
\begin{proposition}\cite[Proposition 1.10.12]{StanleyEnumComb_2012}\label{prop:stanley_num_comp_flags}
    Let $V$ be a $t$-dimensional vector space over $\mathbb{F}_q$. Then the number of complete flags in $V$ is $[t]_q!.$
\end{proposition}
\subsection{The Bipartite Compatibility Graph}
This graph, which we denote by $\mathbb{C}(k,q)$, is defined in \cite{vanbatenburg2026asymptoticallyattainingmoorebound}. The bipartite compatibility graph is the starting point for both $H_{k,q}$ construction in ~\cite{vanbatenburg2026asymptoticallyattainingmoorebound} and our construction. Beyond that, both constructions are different.

Partial flags of some rank set $S$ are maximal chains in an \emph{$S$-rank-selected subposet of $P$}, where $P$ is the lattice of subspaces of $V$, ordered by inclusion and ranked by dimension (see \cite[Section 3.13]{StanleyEnumComb_2012}). In \cite{vanbatenburg2026asymptoticallyattainingmoorebound}, Cames van Batenburg and Korsky distinguish these particular flags and refer to them as \emph{odd-/even-rank partial flags}. We use the definition from \cite{vanbatenburg2026asymptoticallyattainingmoorebound}, except that we include $\{0\}$ and $V$ as flag elements.
\begin{definition}[Partial Flags of Odd-/Even-Rank Type] \cite{vanbatenburg2026asymptoticallyattainingmoorebound} 
    Let $V$ be a $t$-dimensional vector space, where $\mathcal{F}(V)$ is the set of complete flags on $V$. We define $$F_{odd}=(W_i)_{\substack{0 \le i \le t\\ i\ \mathrm{odd}}}, \text{ and }F_{even}=(W_i)_{\substack{0 \le i \le t\\ i\ even}}.$$  Then the set of \emph{partial flags of odd-rank type} in $V$ is $\mathcal{P}_{odd}(V) = \{F_{odd}: F \in \mathcal{F}(V)\}$ and the set of \emph{partial flags of even-rank type} in $V$ is $\mathcal{P}_{even}(V) = \{F_{even}: F \in \mathcal{F}(V)\}.$
\end{definition}
\begin{definition}[Compatible Flags]\cite{vanbatenburg2026asymptoticallyattainingmoorebound}
    An odd flag $F_{odd}$ and an even flag $F_{even}$ of $V$ are said to be \emph{compatible} if when interleaved by rank, they produce a complete flag on $V$.
\end{definition}
\begin{definition} [Bipartite Compatibility Graph]\cite{vanbatenburg2026asymptoticallyattainingmoorebound}
    Let $V=\mathbb{F}_q^{2k+1},$ where $q$ is a prime power and $k \in \mathbb{Z}_{> 0}.$ The \emph{bipartite compatibility graph} of $V$ is the bipartite graph with vertex set $\mathcal{P}_{odd}(V) \sqcup \mathcal{P}_{even}(V)$, and edges $(F_a, F_b)$ when $F_a \in \mathcal{P}_{odd}(V)$, $F_b \in \mathcal{P}_{even}(V)$ and $F_a$ and $F_b$ are compatible. We denote this graph by $\mathbb{C}(k,q).$
\end{definition}

Proposition \ref{prop:order_deg_bcg} follows immediately from the definitions above and Proposition~\ref{prop:stanley_num_comp_flags}. Cames van Batenburg and Korsky also make these arguments in \cite[Proposition 3.1]{vanbatenburg2026asymptoticallyattainingmoorebound}.
\begin{proposition}\label{prop:order_deg_bcg}
    The graph $\mathbb{C}(k,q)$ is $(q+1)^k$-regular, and in particular, $\Delta(\mathbb{C}(k,q))=(q+1)^k.$ Its order is $\frac{2[2k+1]_q!}{(q+1)^k}.$
\end{proposition}
\begin{proof}
    Let $F_{even} = (\{0\} = W_0 \subset \dots \subset W_{2k})$ be a partial flag of even-rank type and let $F_{odd} = (W_1 \subset \dots \subset W_{2k+1}=V).$ $F_{odd}$ is compatible with $F_{even}$ if and only if $W_{2i} \subsetneq W_{2i+1} \subsetneq W_{2i+2},$ for all $0 \le i \le k-1.$ For each $i$, the quotient $W_{2i+2}/W_{2i}$ has dimension $2$, so the possible $W_{2i+1}$ correspond to the $1$-dimensional spaces in the quotient. There are $q+1$ of these, so there must be $(q+1)^k$ independent choices for $F_{odd}$, so each element of $\mathcal{P}_{even}(V)$ has degree $(q+1)^k.$ By a similar argument, each element of $\mathcal{P}_{odd}(V)$ has degree $(q+1)^k.$ So $\mathbb{C}(k,q)$ is $(q+1)^k$-regular. 
    
    The number of edges in $\mathbb{C}(k,q)$ is the same as the number of complete flags in $V$, and by Proposition~\ref{prop:stanley_num_comp_flags}, the number of complete flags in $V$ is $[2k+1]_q!$ Since $\mathbb{C}(k,q)$ is $(q+1)^k$-regular, the degree-sum formula gives $|\mathbb{C}(k,q)| = \frac{2[2k+1]_q!}{(q+1)^k}.$
\end{proof}

\begin{definition}[Rank Function] \cite{vanbatenburg2026asymptoticallyattainingmoorebound}\label{def:rank_function}
    Let $F$ be a flag in $V$ and let $e \in V.$ The \emph{rank function} is the largest rank of any subspace in $F$ not containing $e$ (and $0$ if every member of $F$ contains $e$):
    $$
    \lambda(e, F)=\max(\{\dim W: W \in F, e \notin W\} \cup \{0\}).$$ 
\end{definition}

\begin{lemma} \cite[adapted from proof of Proposition 3.1]{vanbatenburg2026asymptoticallyattainingmoorebound}\label{lemma:dist_rank}
    Let $F, F' \in \mathbb{C}(k,q),$ and let $\lambda$ be the rank function of Definition~\ref{def:rank_function}.  If $\lambda(e,F)=d_1$ and $ \lambda(e,F')=d_2$ for some $e \in V,$ then the shortest path between $F$ and $F'$ in $\mathbb{C}(k,q)$ has length at least $\lvert d_1 - d_2 \rvert.$ \end{lemma}
    \begin{proof}      
        If $F$ and $F'$ are adjacent in $\mathbb{C}(k,q)$, interleaving their subspaces gives a complete flag
        $
        ({0} = W_0 \subset W_1 \subset \ldots \subset W_{2k+1} = V).
        $
        For any $e\in V$, let $d$ be the smallest dimension such that $e\in W_{d}$. Since $d$ is smallest, $e\notin W_{i}$ for $i<d$ and since subspaces in a complete flag are nested, $e\in W_{i}$ for $i\geq d$.
        So for any $e$, $\lvert \lambda(e,F) -\lambda(e,F')\rvert\leq 1.$ Let $F = F_0, F_1, \dots, F_m = F'$ be a path of length $m$ between $F$ and $F'.$ Then by the triangle inequality, $\lvert d_1 - d_2 \rvert \le \sum_{i=1}^{m} \lvert \lambda(e,F_{i-1}) -\lambda(e,F_{i})\rvert \le m.$  
    \end{proof}

 \begin{theorem}\label{th:bi_comp_diam}
    Let $V=\mathbb{F}_q^{2k+1},$ where $q$ is a prime power and $k \in \mathbb{Z}_{>0}$ and let $\mathbb{C}(k,q)$ be the bipartite compatibility graph on $V$. Then the diameter of $\mathbb{C}(k,q)$ is $2k+1$.
\end{theorem}
\begin{proof}
    For the upper bound in this proof, we reproduce the argument in \cite[Lemma 2.1]{vanbatenburg2026asymptoticallyattainingmoorebound}. It is a standard fact (see e.g. \cite[proof of Axiom B1]{AbramenkoBrown2008}) that any two complete flags $F$ and $F'$ are induced by two orderings of a common basis $\{e_1, \dots, e_{2k+1}\},$ where $W_i=\text{span}(e_1, \dots, e_{i})$ and $W_i'=\text{span}(e_{\pi(1)}, \dots, e_{\pi(i)}),$ where $1\le i \le 2k+1$ and $\pi \in S_{2k+1}.$ Applying odd-even sorting to $\pi$, as in \cite[Section $2$]{AlonChungGraham1994}, gives a route between $F$ and $F'$ that changes only even-rank subspaces on even steps and only odd-rank subspaces on odd steps, so this route provides a path between  $F_{even}$ and $F_{odd}'$  of length at most $2k+1$. The same argument gives a path of length at most $2k$ between $F_{odd}$ and $F_{odd}'$ and between $F_{even}$ and $F_{even}',$ so $D(\mathbb{C}(k,q))\le 2k+1.$

    We now extend the result from \cite{vanbatenburg2026asymptoticallyattainingmoorebound} by finding the exact diameter.
    Let $e \ne 0$ and let $F_{odd}'$ be an odd partial flag $(\langle e \rangle \subset \dots \subset V)$ and let $F_{even}$ be an even partial flag $(\{0\} \subset \dots \subset W),$ where $V=W \oplus \langle e\rangle.$ Clearly, $\lambda(e,F_{even})=2k$ and $\lambda(e,F_{odd}')=0$. Then by Lemma~\ref{lemma:dist_rank}, the shortest path between $F_{odd}'$ and $F_{even}$ has length at least $2k.$ However, $F_{even}$ and $F_{odd}'$ are on the opposite sides of the bipartite $\mathbb{C}(k,q)$ so any path between them must have an odd number of edges. So $D(\mathbb{C}(k,q))\ge 2k+1,$ and this, with the upper bound, gives $D(\mathbb{C}(k,q))= 2k+1.$
\end{proof}
\subsection{Polarity on the Bipartite Compatibility Graph}\label{sec:bcg_polarity}
In this section, we show a natural polarity on the bipartite compatibility graph.
\begin{lemma}\label{lemma:perp_flag}
    Let $V=\mathbb{F}_q^{2k+1},$ where $q$ is a prime power and $k \in \mathbb{Z}_{>0},$ let $F=(W_1 \subset \dots \subset W_m)$ be a flag on $V$, and let $B$ be a nondegenerate symmetric bilinear form on $V$.
    Then $F^\perp=(W_m^\perp \subset \dots \subset W_1^\perp)$ is also a flag on $V$, where $W_i^\perp$ is the set of all vectors in $V$ that are orthogonal to $W_i$ with respect to $B$. Moreover, $F$ is a complete flag if and only if $F^\perp$ is a complete flag, and $F$ is a flag of even-rank type if and only if $F^\perp$ is a flag of odd-rank type.
\end{lemma}
\begin{proof}
    Let $W_i$ and $W_j$ be subspaces occurring in $F,$ where $W_i \subset W_j.$ Then $W_i^\perp$ and $W_j^\perp$ are in $F^\perp$ and $W_j^\perp \subseteq W_i^\perp.$ $B$ is nondegenerate, so the dimension of $W_j^\perp$ is $2k+1-\dim W_j,$ and the dimension of $W_i^\perp$ is $2k+1-\dim W_i.$ So the inclusion is proper and $W_j^\perp \subset W_i^\perp.$ A flag $F$ is complete if and only if $F$ contains subspaces of every dimension $\le 2k+1,$  if and only if $F^\perp$ contains subspaces of every dimension $\le 2k+1,$ if and only if $F^\perp$ is a complete flag. Finally $W_i$ has even dimension if and only if $W_i^\perp$ has odd dimension, so the orthogonal complement interchanges partial flags of odd-rank type and partial flags of even-rank type.
\end{proof}
\begin{theorem}\label{th:bi_comp_polarity}
    Let $V=\mathbb{F}_q^{2k+1},$ where $q$ is a prime power and $k \in \mathbb{Z}_{>0}.$ Let $B$ be a nondegenerate  symmetric bilinear form on $V$ and let $\mathbb{C}(k,q)$ be the bipartite compatibility graph on $V$. Then the orthogonal-complement mapping on the vertices of $\mathbb{C}(k,q)$ with respect to $B$ is a polarity.
\end{theorem}
\begin{proof}
    The orthogonal complement of a flag is an involution. $V$ is of odd dimension, so the orthogonal complement of an odd-dimensional subspace $W_i$ has even dimension $2k+1-i$, and vice versa. The orthogonal complement is thus a bijective mapping interchanging the partitions $\mathcal{P}_{odd}(V)$ and $\mathcal{P}_{even}(V)$ of $\mathbb{C}(k,q)$. So the orthogonal complement is a bipartition-swapping involution.
    Let $F$ be a complete flag on $V$, where $F_{odd}$ is its partial subflag of odd-rank type and $F_{even}$ is its partial subflag of even-rank type. By Lemma~\ref{lemma:perp_flag}, $F^\perp$, the orthogonal complement of $F$, is also a complete flag. Under orthogonal complementation,  
    $$
    F_{odd} \rightarrow (F_{odd})^\perp = (F^\perp)_{even}
    $$ and 
    $$
    F_{even} \rightarrow (F_{even})^\perp = (F^\perp)_{odd}.
    $$ So orthogonal complementation on $\mathbb{C}(k,q)$ preserves edges, so is an automorphism.
\end{proof}

\begin{definition}[Orthogonal-Complement Polarity and Quotient Graph]
    The orthogonal-complement polarity $\beta$ on $\mathbb{C}(k,q)$ is the orthogonal-complement mapping with respect to $B$ on the vertices of $\mathbb{C}(k,q)$, i.e., the partial flags of even- and odd-rank type. The polarity quotient graph of $\mathbb{C}(k,q)$ with respect to $\beta$ is denoted by $P_{\beta}(\mathbb{C}(k,q)).$
\end{definition}
\begin{theorem}\label{th:bi_comp_pol_diam}
    Let $V=\mathbb{F}_q^{2k+1},$ where $q$ is a prime power and $k \in \mathbb{Z}_{>1}$ and let $\mathbb{C}(k,q)$ be the bipartite compatibility graph on $V$, admitting the orthogonal complement polarity $\beta$. Then the diameter of $P_{\beta}(\mathbb{C}(k,q))$ is $2k$.
\end{theorem}
\begin{proof} 
    Since $B$ is nondegenerate and symmetric and dim$(V)$ is odd, $V$ admits a $B$-orthogonal basis $\{e_1, e_2, \dots, e_{2k}, e_{2k+1}=v\}$ (see e.g.~\cite[Lemma 11.38]{Taylor1992} when $q \ne 2$, and \cite[Lemma 11.40 and Exercise 11.1(iii)]{Taylor1992} for $q = 2$). Nondegeneracy implies that 
    $B(e_i,e_i) \ne 0$ for every vector $e_i$ in this basis.  This means that if $I \subseteq \{1,2, \dots, 2k+1\}$, $$\langle e_i \mid i \in  I\rangle ^\perp = \langle e_j \mid  j \notin  I\rangle.$$
    
    We construct a partial flag of even-rank type
    \begin{align*}
        F_{even} &= (\{0\} \subset \langle e_1, e_2\rangle \subset \dots \subset \langle e_1, e_2, \dots, e_{2k-2}\rangle \subset \langle e_1, e_2, \dots, e_{2k}\rangle).\\
    \intertext{and a partial flag of odd-rank type by adjoining $v$ to each member of $F_{even}:$}
        F_{odd} &= (\langle v\rangle \subset \langle v, e_1, e_2\rangle \subset \dots \subset \langle v, e_1, e_2, \dots, e_{2k-2}\rangle \subset V).\\
    \intertext{Since the orthogonal-complement polarity $\beta$ maps a flag to its orthogonal-complement flag,}
        \beta(F_{even})&= (\langle v \rangle \subset \langle v, e_{2k-1}, e_{2k}\rangle \subset \dots \subset \langle v,e_3, e_4, \dots, e_{2k} \rangle \subset V), \text{ and}\\
        \beta(F_{odd}) &= (\{0\} \subset \langle e_{2k-1}, e_{2k}\rangle \subset \dots \subset \langle e_3,e_4, \dots, e_{2k}\rangle \subset \langle e_1, \dots, e_{2k}\rangle).
    \end{align*}
    These partial flags have the following rank functions: $\lambda(v,F_{even}) = 2k$ and $\lambda(v,F_{odd}) = 0,$ and $\lambda(e_1,F_{odd}) = 1$ and $\lambda(e_1,\beta(F_{even})) = 2k-1.$
    By Lemma~\ref{lemma:dist_rank}, 
    $$
        \text{dist}(F_{even},F_{odd}) \ge 2k,
    $$
    and since $\beta$ is a polarity and hence an automorphism, $$
        \text{dist}(\beta(F_{even}),\beta(F_{odd})) \ge 2k.
    $$

    Let $(F_{odd}=F_0, F_1, \dots F_{m-1}, F_m = \beta(F_{even}))$ be a path of length $m$ between $F_{odd}$ and $\beta(F_{even}),$ and let $U$ be the subspace of rank $2$ in $F_1$. $F_1$ is compatible with $F_0,$ so $\langle v \rangle \subset U \subset \langle v,e_1, e_2 \rangle.$ Both $U$ and $\langle e_1, e_2 \rangle$ are $2$-dimensional subspaces of the $3$-dimensional $\langle v,e_1, e_2 \rangle,$ so dim$(U \cap \langle e_1, e_2 \rangle) \ge 1.$ Let $w \in U \cap \langle e_1, e_2 \rangle,$ where $w \ne 0.$  Then $\lambda(w,F_1) = 0,$ since $w \in U,$ and $\lambda(w,F_m) = 2k-1:$ $w \in \langle e_1, e_2 \rangle$ so $w \notin \langle v, e_3,e_4, \dots e_{2k}\rangle.$ Lemma~\ref{lemma:dist_rank} then gives that the shortest path between $F_{odd}$ and $\beta(F_{even})$ must have length at least $2k,$ so 
    $$
        \text{dist}(F_{odd},\beta(F_{even})) \ge 2k.
    $$
Again, since $\beta$ is a polarity,     $$
        \text{dist}(F_{even},\beta(F_{odd})) \ge 2k.
    $$
    
    Since every quotient path lifts to a path between representatives, $D(P_{\beta}(\mathbb{C}(k,q))) \ge 2k$, and from Proposition~\ref{prop:polarity_props}(2) and Theorem~\ref{th:bi_comp_diam}, we have that $D(P_{\beta}(\mathbb{C}(k, q))) \le 2k$, so  $D(P_{\beta}(\mathbb{C}(k,q))) = 2k.$
\end{proof}
Theorems \ref{th:bi_comp_diam} and \ref{th:bi_comp_pol_diam} together show that the bipartite graphs $\mathbb{C}(k,q)$ meet the conditions of Corollary~\ref{cor:ulb} and can be used in Kronecker product constructions to obtain graphs of diameter $D(P_{\beta}(\mathbb{C}(k,q))).$
\begin{proposition}\label{prop:bi_comp_pol_order}
    The order of $P_{\beta}(\mathbb{C}(k,q))$ is $$| V(P_{\beta}(\mathbb{C}(k,q)))|=|V(\overline{P_{\beta}(\mathbb{C}(k,q))})|=\frac{[2k+1]_q!}{(q+1)^k}.$$
\end{proposition}
\begin{proof}
    By  Proposition~\ref{prop:polarity_props}($3$), the order of a polarity quotient of a graph is half the order of the graph,  so this follows from Proposition \ref{prop:order_deg_bcg}. 
\end{proof}
\begin{proposition}\label{prop:bi_comp_pol_deg}
    The degree $\Delta(\overline{P_{\beta}(\mathbb{C}(k,q))})=(q+1)^k.$
\end{proposition}
\begin{proof}
    Assume that all of the vertices of $\mathbb{C}(k,q)$ are absolute with respect to the orthogonal-complement polarity. Then every partial flag of odd-rank type $F_{odd}$ is compatible with its orthogonal complement $(F_{odd})^\perp,$ and their union is a complete flag. For any nonzero vector $e \in V,$ we may construct an $F_{odd}$ with $W_1=\langle e \rangle$. Compatibility then shows that $\langle e \rangle \subset \langle e \rangle ^\perp,$ so $B(e,e)=0.$ This also holds trivially for $e=0$. $B$ is thus a nondegenerate alternating bilinear form on $V,$ so the dimension of $V$ must be even (see e.g. \cite[Section~8.2]{Taylor1992}). This contradicts the fact that dim $V=2k+1$.
    
    So $\mathbb{C}(k,q)$ must have at least one non-absolute point with respect to the polarity, and by Lemma~\ref{lemma:polarity_props_diam_equality}, $\Delta(\overline{P_{\beta}(\mathbb{C}(k,q))})= \Delta(\mathbb{C}(k,q)) = (q+1)^k.$
\end{proof}
\subsection{Families of Graphs of Diameter \texorpdfstring{$2k$}{} Approaching the Moore Bound} 
Here, we prove Theorems~\ref{th:main_polarity_bcg} and \ref{th:main_bcg_kronecker_moore}, giving the parameters of the even-diameter polarity quotients of the bipartite compatibility graphs and their Kronecker products, and showing that they approach the Moore bound asymptotically.

\begin{proof}[Proof of Theorem~\ref{th:main_polarity_bcg}]
    The diameter $2k$ follows from Theorem~\ref{th:bi_comp_pol_diam}, the degree $(q+1)^k$ from Proposition~\ref{prop:bi_comp_pol_deg} and the order from Proposition~\ref{prop:bi_comp_pol_order}.

        By Definition~\ref{def:q_factorial_t}, we have that $$[2k+1]_q!=(1+q)(1+q+q^2) \dots (1+q+\dots +q^{2k}),$$ so by Propositions~\ref{prop:bi_comp_pol_order} and \ref{prop:limits_MB},
    \begin{align*}
        \lim_{q\rightarrow \infty} \frac{|\overline{P_{\beta}(\mathbb{C}(k,q))}|}{MB((q+1)^k,2k)} 
        &= \lim_{q\rightarrow \infty} \frac{[2k+1]_q!}{(q+1)^k\cdot MB((q+1)^k, 2k)}\\
        &=\lim_{q\rightarrow \infty} \left(\frac{[2k+1]_q!}{(q+1)^{k(2k+1)}}\right) \cdot \lim_{q\rightarrow \infty} \left(\frac{((q+1)^k)^{2k}}{MB((q+1)^k,2k)}\right)\\
        &=1.
    \end{align*}
so the graphs $\overline{P_{\beta}(\mathbb{C}(k,q))}$ approach the Moore bound as $q \rightarrow \infty.$
\end{proof}

\begin{proof}[Proof of Theorem~\ref{th:main_bcg_kronecker_moore}]
    By Theorem~\ref{th:bi_comp_diam}, the diameter $D(\mathbb{C}(k,q)) = 2k+1$  and by Theorem~\ref{th:bi_comp_pol_diam}, $D(P_{\beta_q}(\mathbb{C}(k,q)))=2k,$ so by Corollary~\ref{cor:ulb}, $D(P_{\beta}(\mathbb{C}(k,q)\odot \mathbb{C}(k,r))) = 2k.$    
    
    By Proposition \ref{prop:order_deg_bcg}, $\Delta(\mathbb{C}(k,q)) = (q+1)^k$ and $\Delta(\mathbb{C}(k,r)) = (r+1)^k.$ So by Proposition~\ref{prop:redkronecker_stats}, $\Delta(P_{\beta}(\mathbb{C}(k,q)\odot \mathbb{C}(k,r)) = (q+1)^k(r+1)^k.$ 

    Recall that by Theorem \ref{th:main_isomorphism}, $P_{\beta}(\mathbb{C}(k,q)\odot \mathbb{C}(k,r) \cong P_{\beta_q}(\mathbb{C}(k,q)) \otimes P_{\beta_r}(\mathbb{C}(k,r))$, with $\beta=(\beta_q, \beta_r)$.
    The order of $P_{\beta_q}(\mathbb{C}(k,q)) \otimes P_{\beta_r}(\mathbb{C}(k,r))$ is then $|P_{\beta_q}(\mathbb{C}(k,q))|\cdot |P_{\beta_r}(\mathbb{C}(k,r))|,$ by the definition of the Kronecker product. This is $$\frac{[2k+1]_q! \cdot [2k+1]_r!}{(q+1)^k(r+1)^k}$$ by Proposition~\ref{prop:bi_comp_pol_order}.
    
    The families $\{\overline{P_{\beta_q}(\mathbb{C}(k,q))}\}$ and $\{\overline{P_{\beta_r}(\mathbb{C}(k,r))}\}$ approach $100$\% of the Moore bound, by Theorem~\ref{th:main_polarity_bcg}. All conditions of Theorem~\ref{th:new_general_MB_fraction_new} are met, so the theorem follows.
\end{proof}

\subsection{Comparison to Halved-Flag Graphs}
In \cite{vanbatenburg2026asymptoticallyattainingmoorebound}, the authors construct halved-flag graphs of any diameter $D$. Our polarity-quotient and Kronecker-product constructions hold only for even diameter $2k$, but the degree coverage of both constructions is better than the even-diameter halved-flag graphs. All three families below have diameter $2k$: \begin{itemize}
    \item The polarity quotients $\overline{P(C(k,q))}$, have degree $(q+1)^k.$ 
    \item The constructions $\overline{P(\mathbb{C}(k,q)\odot \mathbb{C}(k,r))}$ have degree $(q+1)^k(r+1)^k,$ equal to $(q+1)^{2k}$ when $q=r,$ and allow all prime-power pairs $(q,r),$ including $q \ne r$.
    \item The halved-flag graphs $H_{2k,q}$ have degree asymptotic to $(q+1)^{4k}.$
\end{itemize}
We essentially trade diameter coverage for degree coverage. This is summarized in Table~\ref{table:halved_vs_ours}, which gives parameters for all $H_{k,q}$, not just those of even diameter. We also compare the diameter-$4$ $\overline{P(C(2,q))}$, $\overline{P(\mathbb{C}(2,q)\odot \mathbb{C}(2,r))}$ and $H_{4,q}$ graphs in Figure~\ref{fig:bcg_compare} as an example. For the degree range and diameter shown, this figure shows the increased degree coverage and higher Moore-bound ratio for the polarity-derived graphs in this paper. 
\begin{table}[!ht]
\caption{Comparison of the halved-flag graphs $H_{k,q}$ from \cite{vanbatenburg2026asymptoticallyattainingmoorebound} to the polarity quotients of bipartite compatibility graphs $\overline{P(\mathbb{C}(k,q))}$ and their Kronecker-product constructions $\overline{P(\mathbb{C}(k,q)\odot \mathbb{C}(k,r))}$. The orders of the $H_{k,q}$ and the $\overline{P(\mathbb{C}(k,q))}$ are the same, given in Proposition~\ref{prop:bi_comp_pol_order} and \cite{vanbatenburg2026asymptoticallyattainingmoorebound}, and for ease of reading are represented here as $\mathcal{O}_{k,q}$.}
 \label{table:halved_vs_ours}
\noindent\adjustbox{max width=\textwidth}
{
\renewcommand{\arraystretch}{1.35}
\begin{tabular}{lccc}
\toprule
\textbf{Construction} & \textbf{Degree} &\textbf{Diameter} & \textbf{Order}  \\

\midrule
$H_{k,q}$ & $\le (q+1)^k((q+1)^k-1) \sim (q+1)^{2k}$& $k$ & $\mathcal{O}_{k,q}$     \\
$\overline{P(\mathbb{C}(k,q))}$  &$(q+1)^k$& $2k$ & $\mathcal{O}_{k,q}$    \\
$\overline{P(\mathbb{C}(k,q)\odot \mathbb{C}(k,r))}$  & $(q+1)^k(r+1)^k$& $2k$ & $\mathcal{O}_{k,q}\mathcal{O}_{k,r}$ \\
 \bottomrule
\end{tabular} 
}
\end{table}

\begin{figure*}[!ht]
\centering
\includegraphics[width=.9\textwidth]{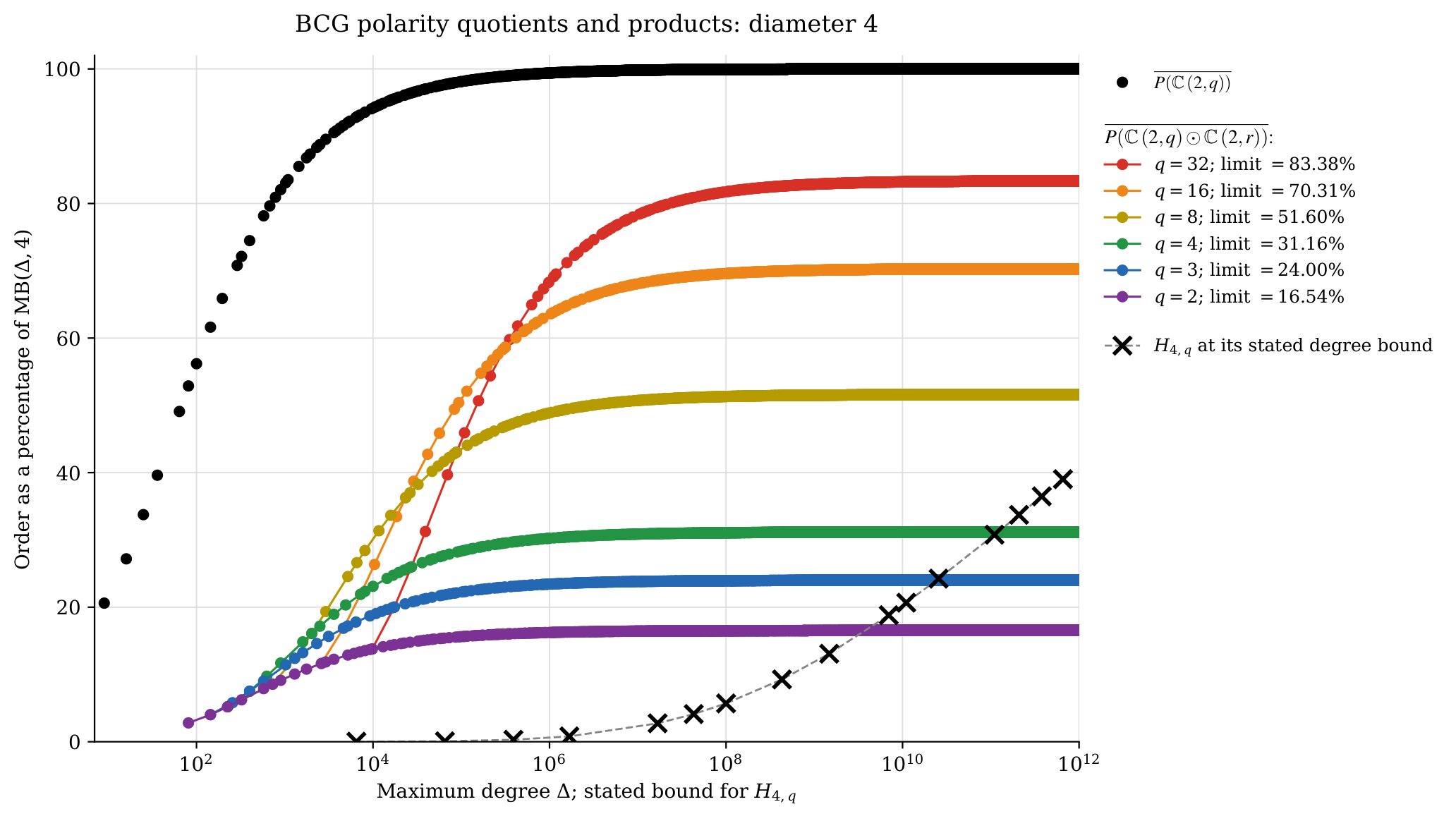}
\caption{Comparing the percentage of the Moore bound met by the bipartite compatibility graph constructions of diameter $4$ in \cite{vanbatenburg2026asymptoticallyattainingmoorebound} and in this paper. The series represent the polarity quotients described in Section~\ref{sec:bcg_polarity}, the $\overline{P(\mathbb{C}_{2,q} \odot \mathbb{C}_{2,r})}$ families for several $q$,  as discussed in Theorems~\ref{th:main_bcg_kronecker_moore} and \ref{th:new_general_MB_fraction_new}, and the $H_{4,q}$ graphs from \cite{vanbatenburg2026asymptoticallyattainingmoorebound}. For the $H_{4,q}$ graphs, we use the degree bound given in \cite{vanbatenburg2026asymptoticallyattainingmoorebound}, giving an asymptotically correct but conservative representation of their Moore-bound ratios. The graphs in this paper cover a broader range of degrees, and achieve substantially larger Moore-bound ratios at much smaller degrees than those from \cite{vanbatenburg2026asymptoticallyattainingmoorebound}. The lines joining the data points serve to distinguish the growth of each family.}
\label{fig:bcg_compare}
\end{figure*}

\section{Structural Results Applied to Generalized Polygons}\label{sec:genpolys_complete}
Here, we apply the results from Section~\ref{sec:main_th} to the special case of generalized polygons to construct new families of large graphs.

\subsection{Background}\label{sec:genpolys} Generalized polygons were introduced by Tits in 1959 in \cite{gen_poly_orig_tits}, and have seen extensive development since then, as discussed in Van Maldeghem's book \cite{GenPolys_VanMaldgehem}, Thas' article in  \cite[Chapter 9]{THAS_gen_polys}, Payne and Thas' book on generalized quadrangles \cite{paynethas} and elsewhere. 
Intuitively, the generalized $n$-gon is such that every point-line pair is included in some ordinary $n$-gon contained in the generalized $n$-gon, and there are no ordinary $m$-gons contained in the generalized $n$-gon for $m<n$. Equivalently, a generalized $n$-gon has an incidence graph of diameter $n$ and girth $2n$ \cite{GenPolys_VanMaldgehem}. 

In this paper, we are concerned with the combinatorial properties of generalized polygons much more than with their geometric structure. 

\begin{definition}[Incidence-graph representation of a point-line geometry] \label{def:incidence_genpoly}
    Let $\Gamma=(\mathcal{P}, \mathcal{L}, \mathbf{I})$ be a point-line geometry, where $\mathcal{P}$ and $\mathcal{L}$ are the sets of points and lines and $\mathbf{I}$ is the incidence relationship between them. An incidence-graph representation $\mathbb{G}(\Gamma)$ of $\Gamma$ is a graph in which $V(\mathbb{G}(\Gamma))=\mathcal{P} \cup \mathcal{L}$, and $(x,y) \in E(\mathbb{G}(\Gamma))$ when $x \in \mathcal{P}$, $y \in \mathcal{L}$ and $x$ lies on $y$, or vice versa.
\end{definition}
\begin{definition}[Order of a generalized $n$-gon, thick $n$-gon]\cite{GenPolys_VanMaldgehem}
A generalized $n$-gon has \emph{order} $(s,t),$ where every line is incident with $s+1$ points and every point is incident to $t+1$ lines. A generalized $n$-gon is \emph{thick} when both $s,t>1,$ and \emph{thin} otherwise.
\end{definition}
 Since a generalized polygon admitting polarity is self-dual, it has generalized-polygon order $(s,t)$ with $s=t$. 

Well-known results of Feit and Higman \cite[Theorem 1]{feit_higman} imply that a thick generalized $n$-gon exists only when $n \in \{2,3,4,6,8\}$. 
Thick generalized $n$-gons $\mathbb{G}_n(q,q)$ admitting polarity are known to exist when $n=2$ and $q \in \mathbb{Z}_{>1},$ $n=3$ and $q$ is a prime power, $n=4$ and $q=2^{2m+1}$ for $m \in \mathbb{Z}_{\ge 0}$, and $n=6$ and $q=3^{2m+1}$ for $m \in \mathbb{Z}_{\ge 0}$, and cannot exist for other $n$ \cite[Theorem 1]{feit_higman},   \cite[Proposition 7.2.7]{GenPolys_VanMaldgehem}.

As mentioned earlier,  the incidence graphs of the Desarguesian generalized triangles $\mathbb{G}_3(q,q)$ are in fact isomorphic to the bipartite compatibility graphs $\mathbb{C}(1,q).$ Such an isomorphism cannot hold for $n$-gons with $n$ even: these have even diameter, whereas each $\mathbb{C}(k,q)$ has odd diameter.

\subsection{Thin Polygons}\label{sec:thin}
Thin generalized $n$-gons admitting polarity are self-dual, so must have generalized-polygon order $(1,1)$. These are then just ordinary $n$-gons, and the incidence graph of the thin generalized $n$-gon $\mathbb{G}_n(1,1)$ is  $C_{2n},$ the cycle on $2n$ vertices. 
\begin{figure}[!ht]
    \centering
    \begin{subfigure}[t]{.31\linewidth}
        \centering      
        \includegraphics[width=.6\linewidth]%
        {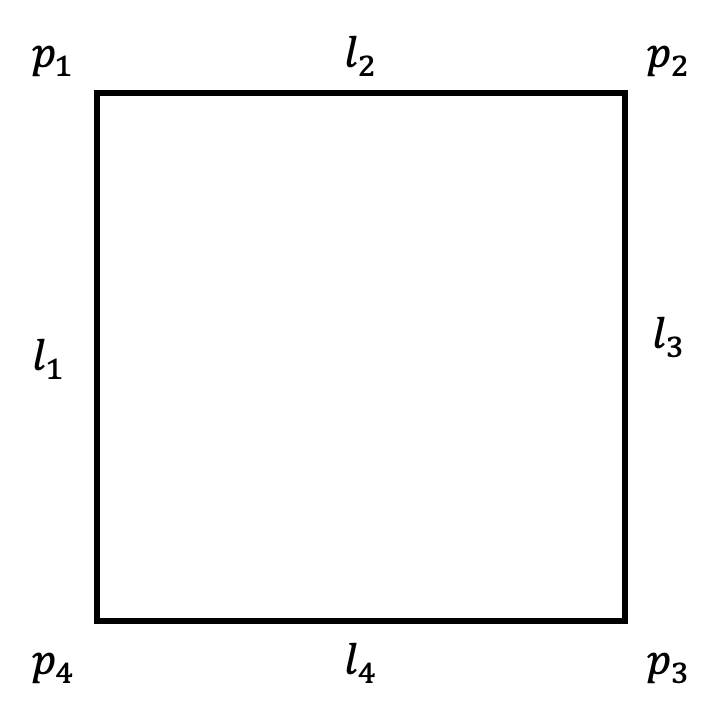}        
        \caption{The generalized quadrangle $\mathbb{G}_4(1,1)$ is represented by a grid of $4$ points $\{p_1, p_2, p_3, p_4\}$ and $4$ lines $\{l_1, l_2, l_3, l_4\}$.}
        \label{fig:g4_1_1}
    \end{subfigure}\hfill
    \begin{subfigure}[t]{.31\linewidth}
        \centering      
        \includegraphics[width=.6\linewidth]
        {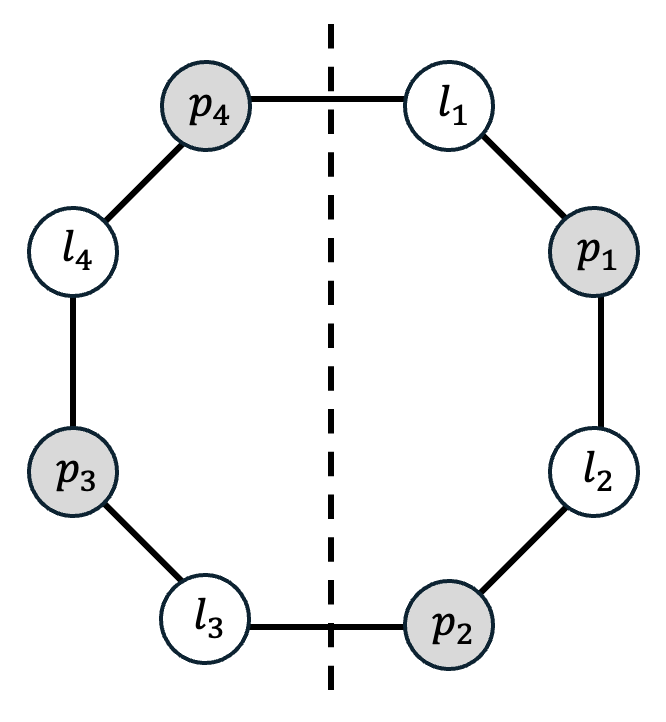}
        \caption{The bipartite incidence graph of $\mathbb{G}_4(1,1)$.  The two partitions are the set of points (grey) and the set of lines (white). The reflection polarity applied in this example is across the dotted-line axis.}  \label{fig:gq_1_1_incidence_new}
    \end{subfigure}\hfill
    \begin{subfigure}[t]{.31\linewidth}
        \centering      \includegraphics[width=.9\linewidth]{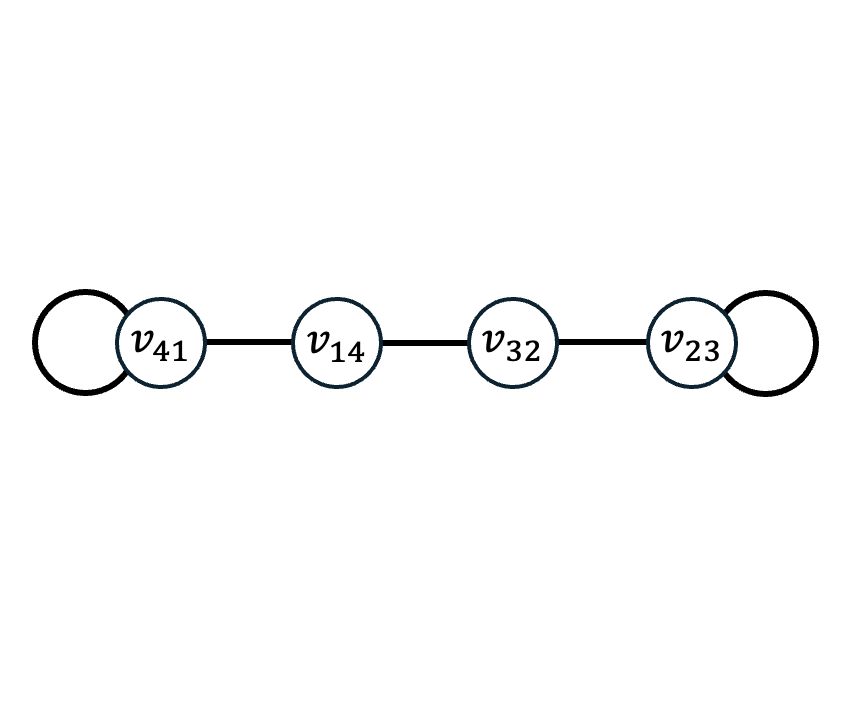}
        \caption{The polarity quotient graph $P(\mathbb{G}_4(1,1))$, where the quotient merges $p_i$ and $\ell_j$ to obtain $v_{ij} = [p_i] = [\ell_j]$. In this example, $n$ is even, so the rotation involution is not a polarity.}
        \label{fig:gq_1_1_polarity}
    \end{subfigure}\hfill
\caption{$\mathbb{G}_4(1, 1)$ and one of its reflection polarities, with its bipartite incidence and polarity quotient graphs.}
\label{fig:gq_1_1_and_incidence_new}
\end{figure}

Proposition~\ref{prop:thin_polarities_new} classifies the polarities on $\mathbb{G}_n(1,1)$, and follows immediately from \cite[Section 4]{bipartition_reversing_involutions_10}, since the incidence graph of the unique thin $n$-gon is $C_{2n}$.
\begin{proposition}\label{prop:thin_polarities_new}
    Let $\mathbb{G}=\mathbb{G}_n(1,1)$ be the unique thin generalized $n$-gon admitting polarity with $n \in \mathbb{Z}_{>1}$. 
    
    $\mathbb{G}$ admits $n$ reflection polarities. The polarity quotient graph with respect to each reflection polarity is $P_n$, the path graph on $n$ vertices with a self-loop at each endpoint, of diameter $n-1.$ 
    
     $\mathbb{G}$ admits one additional rotation polarity if and only if $n$ is odd, rotating the vertices by $\pi$ radians. The polarity quotient graph with respect to the rotation polarity is $C_n$, the cycle graph on $n$ vertices with no self-loops, of diameter $\frac{n-1}{2}.$ 
    
    There are no further polarities on $\mathbb{G}$.
\end{proposition}

Proposition~\ref{prop:thin_polarities_new} shows that the diameter condition of Corollary~\ref{cor:ulb} fails for the rotation polarity. On the other hand, it succeeds for any reflection polarity. 
We use Corollary~\ref{cor:ulb}, applied to thin polygons, to construct some of the large graphs in subsequent sections. We therefore use only the reflection polarity for thin generalized polygons in this paper. 

To illustrate the concepts in Sections~\ref{sec:genpolys} and \ref{sec:thin}, the thin generalized polygon $\mathbb{G}_4(1,1)$ is shown in Figure~\ref{fig:gq_1_1_and_incidence_new}, with its incidence graph and polarity quotient with respect to one of the reflection polarities.

\subsection{Prior Results on Generalized Polygons}
The following four propositions are well-known for thick generalized polygons~\cite{GenPolys_VanMaldgehem}.
After applying Proposition \ref{prop:thin_polarities_new}, these propositions hold for the thin generalized polygons with respect to the reflection polarity as well. We state them here in this slightly more general context as thin polygons have been useful in the construction of a new largest graph in Section~\ref{sec:new_graphs}.
\begin{proposition}~\cite[Lemma 1.5.4]{GenPolys_VanMaldgehem}\label{prop:polygon_stats}
    Let $\mathbb{G} = \mathbb{G}_n(q,q)$ be a generalized $n$-gon. Then
    \begin{enumerate}
        \item $\Delta(\mathbb{G}) = q+1,$ and
        \item $D(\mathbb{G}) = n.$
    \end{enumerate}
\end{proposition}

\begin{proposition}~\cite[Lemma 1.5.4]{GenPolys_VanMaldgehem}\label{prop:polygon_polarity_stats_verts}
        Let $\mathbb{G} = \mathbb{G}_n(q,q)$ be a  generalized $n$-gon admitting polarity, with $P(\mathbb{G})$ its polarity quotient graph. Then 
        $|V(P(\mathbb{G}))| = |V(\overline{P(\mathbb{G})})| = \sum_{i=0}^{n-1}q^i.$\end{proposition}

Proposition~\ref{prop:abs_pts_all_new}(\ref{abs_pts_thick_3}) is presented in \cite[Proposition 7.2.8]{GenPolys_VanMaldgehem} and Proposition~\ref{prop:abs_pts_all_new}(\ref{abs_pts_thick_even}) in  \cite[Propositions 7.2.3 and 7.2.5] {GenPolys_VanMaldgehem}.  Proposition~\ref{prop:abs_pts_all_new}(\ref{abs_pts_thin_ref}) and (\ref{abs_pts_thin_rot}) are discussed in \cite[Section 4]{bipartition_reversing_involutions_10}.
\begin{proposition}\label{prop:abs_pts_all_new}
    Let $\mathbb{G} = \mathbb{G}_n(q,q)$ be a generalized $n$-gon admitting a polarity $\rho$, and let $P(\mathbb{G})$ be its polarity quotient graph. Let $N_\rho$ be the number of absolute vertices in the polarity quotient graph with respect to $\rho$. Then
    $N_\rho$ satisfies: 
    \begin{enumerate}[itemsep=2pt, parsep=1pt]
        \item{\makebox[4.5cm][l]{$q+1 \le N_\rho \le q\sqrt q+1,$} when $\mathbb{G}$ is thick and $n=3,$}\label{abs_pts_thick_3}
        \item{\makebox[4.5cm][l]{$N_{\rho} = q^{\frac{n}{2}}+1$} when $\mathbb{G}$ is thick and $n \in\{2,4,6\},$}\label{abs_pts_thick_even}
        \item{\makebox[4.5cm][l]{$N_{\rho} = 2$} when $\mathbb{G}$ is thin and $\rho$ is a reflection polarity,} and\label{abs_pts_thin_ref}
        \item{\makebox[4.5cm][l]{$N_{\rho} = 0$} when $\mathbb{G}$ is thin, $n$ is odd and $\rho$ is the rotation polarity.} \label{abs_pts_thin_rot}
    \end{enumerate}
    This accounts for every polarity on any generalized $n$-gon admitting polarity. 
\end{proposition}
Proposition~\ref{prop:non-abs_pts_new} then follows from Propositions~\ref{prop:polygon_polarity_stats_verts} and \ref{prop:abs_pts_all_new}.
\begin{proposition}\label{prop:non-abs_pts_new}
    Let $\mathbb{G} = \mathbb{G}_n(q,q)$ be a generalized $n$-gon admitting a polarity $\rho$. Then $\mathbb{G}$ has non-absolute points with respect to $\rho$ if and only if $n>2$. 
\end{proposition}

Proposition~\ref{prop:polarity_gen_polys_diam_new} gives the exact diameter of the polarity quotient of a generalized $n$-gon admitting polarity. This observation appears in several early papers, but to our knowledge, Loz and Pineda-Villavicencio were the first to provide a proof for thick $n$-gons \cite[Theorem 2.7]{lozpineda2010}. Using Proposition~\ref{prop:thin_polarities_new}, we extend this result to thin polygons with reflection polarity.
\begin{proposition}
\label{prop:polarity_gen_polys_diam_new}
    Let $\mathbb{G}$ be a
    generalized $n$-gon admitting polarity. Let $P(\mathbb{G})$ be the polarity quotient graph with respect to a non-rotation polarity of $\mathbb{G}$. Then $D(P(\mathbb{G})) = D(\mathbb{G})-1$. \end{proposition}

We use Proposition~\ref{prop:polarity_gen_polys_diam_new} and Lemma~\ref{lemma:polarity_props_diam_equality} to establish the results in this paper and to construct new largest-known graphs.
We develop the generalized quadrangle example in Section~\ref{sec:new_graphs}, illustrating the case where we have found a new largest-known graph of diameter $3$ and degree $18$.

\subsection{New Results on Generalized Polygons}\label{sec:main_genpolys}
We apply the general theorems of Section~\ref{sec:main_bipartite} to generalized polygons having the same diameter and derive asymptotic values for the number of vertices in graphs in these families. The advantage of using these polygons in this manner is in the fact that their polarity quotients asymptotically approach the Moore bound themselves, as shown in ~\cite{delorme_french_polys_replication_85}. Since they satisfy the theorems and corollaries in Section~\ref{sec:main_bipartite}, the Kronecker products of their polarity quotients will also approach the Moore bound asymptotically, as both field orders $q$ and $r$ tend to infinity.

\begin{corollary}
\label{cor:main_diameter}
        Let $\mathbb{G}_m=\mathbb{G}_m(q,q)$ be a generalized $m$-gon and  $\mathbb{G}_n=\mathbb{G}_n(r,r)$ be a generalized $n$-gon, both admitting polarities, and let $P_m(\mathbb{G}_m)$ and $P_n(\mathbb{G}_n)$ be their polarity quotient graphs with respect to non-rotation polarities.          Then the diameter of  $P(\mathbb{G}_m \odot \mathbb{G}_n)$,  the resulting composed-polarity quotient graph on $\mathbb{G}_m \odot \mathbb{G}_n,$ is $D=\max(m,n)-1.$  
\end{corollary}

\begin{proof}
    By Proposition~\ref{prop:polygon_stats} ($2$), $\mathbb{G}_m$ has diameter $m$ and $\mathbb{G}_n$ has diameter $n$, and by Proposition \ref{prop:polarity_gen_polys_diam_new},
     $D(P_m(\mathbb{G}_m))=m-1$ and $D(P_n(\mathbb{G}_n))=n-1.$
    The diameter follows from Corollary~\ref{cor:ulb}.  \end{proof}

\begin{lemma}\label{lemma:stats_gen_poly_kron_nge3}
    Let $\mathbb{G}_m(q,q), \mathbb{G}_n(r,r), P_m(\mathbb{G}_m) \otimes P_n(\mathbb{G}_n)$ and $P(\mathbb{G}_m \odot \mathbb{G}_n)$ be as in Corollary~\ref{cor:main_diameter}, where at least one of $m,n \ne 2,$ and let $ G= P(\mathbb{G}_m \odot \mathbb{G}_n)$.  Then  
    \begin{enumerate}[itemsep=2pt,parsep=0pt]
        \item $|V(\overline{G})|=(\sum_{i=0}^{m-1}{q^i})(\sum_{i=0}^{n-1}{r^i}),$ and
        \item $\Delta(\overline{G})=        		    (q+1)(r+1)$. 
    \end{enumerate}
\end{lemma}

\begin{proof}

Theorem \ref{th:main_isomorphism} implies that $P(\mathbb{G}_m \odot \mathbb{G}_n) \cong P_m(\mathbb{G}_m) \otimes P_n(\mathbb{G}_n).$    Clearly, graphs obtained by dropping all self-loops from isomorphic graphs are themselves isomorphic.\\
    1) We use the $P_m(\mathbb{G}_m(q,q)) \otimes P_n(\mathbb{G}_n(r,r))$ construction to get the number of vertices. Removing the self-loops at this point has no effect on the number of vertices, so by Proposition \ref{prop:polygon_polarity_stats_verts} and Proposition \ref{prop:kronecker_props}, the number of vertices is  $$|V(\overline{G})|= |V(G)| = \left(\sum_{i=0}^{m-1}{q^i}\right)\left(\sum_{i=0}^{n-1}{r^i}\right).$$

    2) We use the $P(\mathbb{G}_m(q,q) \odot \mathbb{G}_n(r,r))$ construction to get the degree. 
    Both $\mathbb{G}_m(q,q)$ and $\mathbb{G}_n(r,r)$ have a bipartition-reversing automorphism, so by Proposition~\ref{prop:redkronecker_stats}(1) and Proposition~\ref{prop:polygon_stats}(1), $\Delta(\mathbb{G}_m(q,q) \odot \mathbb{G}_n(r,r)) = (q+1)(r+1).$ $\mathbb{G}_m(q,q) \odot \mathbb{G}_n(r,r)$ is bipartite.
    Since $\mathbb{G}_m(q,q) \odot \mathbb{G}_n(r,r)$ is regular, all its vertices are of maximal degree $(q+1)(r+1).$ Since at least one of $m,n \ne 2$, by Proposition~\ref{prop:non-abs_pts_new}, there is a non-absolute vertex $v$ in one of the factors, so for any $w$ in the other factor, $(v,w)$ is non-absolute with respect to the composed polarity. 
    So by Lemma~\ref{lemma:polarity_props_diam_equality},   $$\Delta(\overline{G}) = (q+1)(r+1).$$
\end{proof}

\begin{figure*}[!ht]
\centering
\includegraphics[width=.9\textwidth]{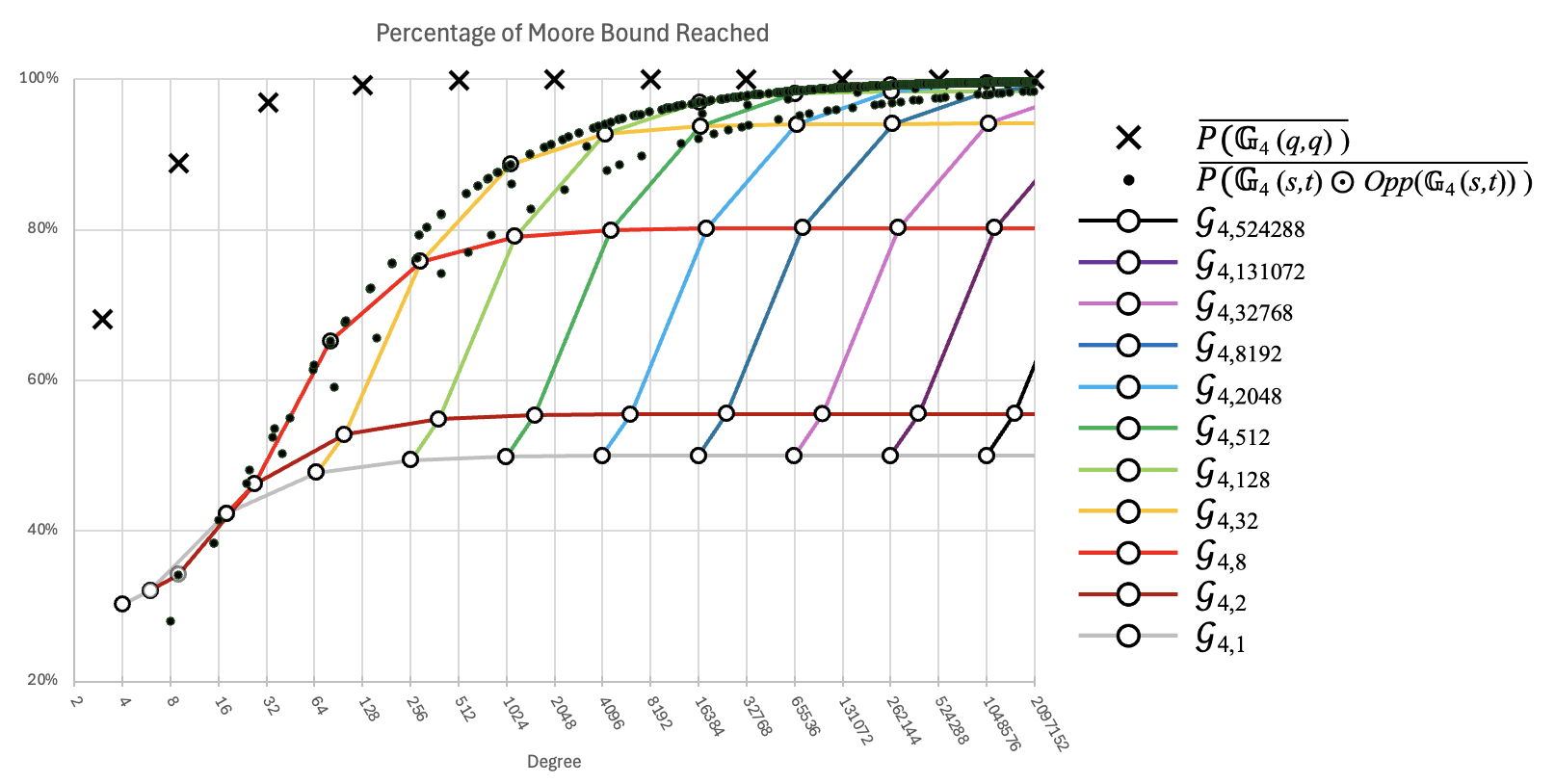}
\caption{Comparing the percentage of the Moore bound met by constructions in \cite{delorme_french_polys_replication_85} and \cite{delorme_opp_85} to the percentage met by the diameter-$3$ generalized-quadrangle constructions $\mathcal{G}_{4,q}$. The $\mathcal{G}_{4,q}$ constructions mostly cover degrees not covered by previously known graphs. The circles in the graph represent actual values. The colored lines serve to indicate the growth of each $\mathcal{G}_{4,q}$ family. The asymptotic limits are as in Theorem~\ref{th:new_general_MB_fraction_new}(2). $\mathcal{G}_4 = \mathcal{G}_{4,1} \cup \mathcal{G}_{4,2} \cup \mathcal{G}_{4,8} \cup \mathcal{G}_{4,32} \cup ...;$ the ratio of the graph orders in $\mathcal{G}_4$ to the Moore bound approaches $1$ as $q,r \rightarrow \infty$. }
\label{fig:delorme_sgq}
\end{figure*} 
\subsection{New Infinite Families of Large Diameter-2, -3 and -5 Graphs and Their Approach to the Moore Bound}\label{sec:main_newfamilies} We discuss in this section constructions from generalized $n$-gons with $n>2$. Fixing $n$, we note that their degrees $(q+1) \rightarrow \infty$ as $q \rightarrow \infty$, so Theorem~\ref{th:new_general_MB_fraction_new} applies, and the graphs in the families $\mathcal{G}_{n}$ approach the Moore bound as $q,r \rightarrow \infty.$ 

\begin{corollary}\label{cor:main_corollary}
     Let $\mathbb{G}_n(q,q)$ and $\mathbb{G}_n(r,r)$ be incidence graphs of generalized $n$-gons, where 
     $n \in \{3,4,6\}$, $q$ and $r$ are such that $\mathbb{G}_n(q,q)$ and $\mathbb{G}_n(r,r)$ admit non-rotation polarities $\rho_q$ and $\rho_r$. Let $\mathcal{G}_{n}$ be the family of all composed-polarity quotient graphs $\overline{P(\mathbb{G}_n(q,q) \odot \mathbb{G}_n(r,r))}.$ Then
    $\mathcal{G}_{n}$
    is an infinite family of graphs of diameter $n-1$ covering degrees $(q+1)(r+1)$, whose graph orders asymptotically approach the Moore bound as both $q$ and $r \rightarrow \infty$.
\end{corollary}

\begin{proof}
The corollary follows from Proposition \ref{prop:polarity_gen_polys_diam_new}, Corollary~\ref{cor:main_diameter} and Theorem~\ref{th:new_general_MB_fraction_new}.   
    \end{proof}

Theorem~\ref{th:new_general_MB_fraction_new} is illustrated in Figure~\ref{fig:delorme_sgq} for diameter-$3$ graphs ($n=4$). The graph orders of the family $\mathcal{G}_{4,2}$ approach $\approx 56\%$ of the Moore bound, the graph orders of the family $\mathcal{G}_{4,8}$ approach $\approx 80\%$ of the Moore bound, etc.

It is easy to show that graphs constructed from quadrangles and hexagons using  Corollary~\ref{cor:main_corollary} exist for more degrees than in Delorme's construction in \cite{delorme_french_polys_replication_85}. The same holds for generalized triangles, but we do not discuss that case here. We restrict this discussion to the known thick families with parameters $q=2^{2a+1}$ for quadrangles and $q=3^{2a+1}$ for hexagons, together with the thin cases where $q=1$.
\begin{itemize}
    \item $\mathcal{G}_4:$
    Thick graphs in this family have degree $(2^k+1)(2^\ell+1) = 2^{k+\ell}+2^k+2^\ell+1$, where $k$ and $\ell$ are odd. Let $k+\ell=2m$. When $m>2$, their degrees fall entirely into the interval $(2^{2m}, 2^{2m+1}).$ This gives $\lceil{\frac{m}{2}}\rceil$ graphs in this degree interval. Including a graph constructed with the thin quadrangle as a factor ($k=0$ and $\ell = 2m-1),$ the construction covers $\lceil{\frac{m}{2}}\rceil+1$ degrees in $(2^{2m}, 2^{2m+1}),$  and no other degrees within $[2^{2m-1}, 2^{2m+1})$, as can be seen in Figure~\ref{fig:delorme_sgq}. (Note that the cases $m=1,2$ are exceptions to this counting rule.) This is in contrast to the discussion in \cite{delorme_french_polys_replication_85}, which produces exactly $1$ graph having degree in $[2^{2m-1}, 2^{2m+1}).$ 
    \item $\mathcal{G}_6:$ 
    Thick graphs in this family have degree $(3^k+1)(3^\ell+1) = 3^{k+\ell}+3^k+3^\ell+1$, where $k$ and $\ell$ are odd. Let $k+\ell=2m.$ Reasoning as in the $\mathcal{G}_4$ case, we see that when $m> 1,$ the construction covers  $\lceil{\frac{m}{2}}\rceil$ degrees in $(3^{2m}, 3^{2m+1})$ (the thick polygons) and one additional degree within $[3^{2m-1}, 3^{2m+1})$ (the thin polygon).  As in the $\mathcal{G}_4$ case, the case $m=1$ is an exception to this counting rule. This is in contrast to the discussion in \cite{delorme_french_polys_replication_85}, which produces exactly $1$ graph having degree in $[3^{2m-1}, 3^{2m+1}).$
\end{itemize}

\subsection{Generalized 2-gons and the Infinite Family \texorpdfstring{$\mathcal{G}_2$}{}}
We complete here the characterization of polarity quotients of Kronecker products of generalized $n$-gons admitting polarity with a short discussion of Kronecker products of $2$-gons. 
The incidence graph of $G_2(q,q)$ is the balanced complete bipartite graph $K_{q+1,q+1}$ for any $q \in \mathbb{Z}_{>0},$ and any bijection from one partition to the other extends to a polarity. All vertices of this graph are absolute under any polarity, so the degree of the quotient of the Reduced-Kronecker product of a pair of $2$-gons is not the same as that of the Reduced-Kronecker product itself. This changes the calculations throughout, so the $n=2$ case must be treated individually. However, it is trivial to show that the composed-polarity graphs of this family are complete, so are Moore graphs. 

\subsection{Replication of Vertices} \label{sec:replication} 
Delorme provides a technique in \cite{delorme_french_polys_replication_85} for replication of a set $S$ in a graph $G$ with pairwise shortest path distance of at least $3$ with $k$ copies of $S$, to obtain a new graph having order $|G|+k|S|,$ degree $\Delta(G)+k$ and diameter $D(G)$, and notes that the set of absolute vertices of the polarity of a generalized quadrangle or hexagon form such a set $S$. Polarity quotients of generalized quadrangles and generalized hexagons both meet the shortest-path condition, so the replication technique may be used on these graphs. We have obtained graphs nearly as large as the largest known using this technique for diameter-$3$ degree-$19$ and -$20$ graphs.

\subsection{Results}
\subsubsection{A New Largest Diam-3 Graph}\label{sec:new_graphs}
Let $\mathbb{G}_4(q,q)$ denote the incidence graph of a 
generalized quadrangle of order $(q,q)$ admitting a polarity.
Let $\overline{G}$ be the polarity quotient graph $\overline{P(\mathbb{G}_4(q_1,q_1) \odot \mathbb{G}_4(q_2,q_2))}$. 
Using Corollary~\ref{cor:main_diameter} and Lemma~\ref{lemma:stats_gen_poly_kron_nge3}, 
for $q_1=1$ and $q_2=8$, we get a diameter-$3$, degree-$18$ graph of order $2340$, which is the largest known of that degree and diameter as of \comellasdate. To our knowledge, the previous largest diameter-$3$, degree-$18$ graph was a 2024 star-product construction of order $1830$ \cite{PolarStar_23}, and before that, the largest diameter-$3$, degree-$18$ graph was of order 1620, cited in  \cite{comb_wiki_degdiam_general}.

\subsubsection{Other Results}
We show this and other results in Figure~\ref{fig:delorme_sgq} and
Table~\ref{table:design_space_new}. 
It can be seen in Figure~\ref{fig:delorme_sgq} that in each interval $[2^a, 2^{a+2})$ with $a$ odd, there is only one $\overline{P(\mathbb{G}_4(q,q))}$ graph \cite{delorme_french_polys_replication_85}, but the number of $\overline{P(\mathbb{G}_4(q,q) \odot \mathbb{G}_4(r,r))}$ graphs increases as $a$ increases.  On the other hand, there are many more $\overline{P(\mathbb{G}(s,t) \odot Opp(\mathbb{G}(s,t)))}$ graphs \cite{delorme_opp_85} in each interval. This is due to the plethora of generalized quadrangles $\mathbb{G}_4(s, t)$ that may be used as the base bipartite graph. In contrast, we use only generalized quadrangles with $q$ and $r$ either $1$ or odd powers of $2$ in the $\overline{P(\mathbb{G}_4(q,q))}$ and $\overline{P(\mathbb{G}_4(q,q) \odot \mathbb{G}_4(r,r))}$ constructions, since these are the quadrangles with polarity that are known to exist. 

\begin{table}[!ht]
\caption{Design space of diameter-3 polarity graphs with self-loops removed, up to degree $20$. Graphs obtained by our construction, using known generalized quadrangles admitting polarity as factors, are in boldface. }
 \label{table:design_space_new}
\noindent\adjustbox{max width=\textwidth}
{
\begin{tabular}{rrrrcc}
\toprule
\textbf{Degree} & \textbf{Order} & \textbf{MB} & \textbf{\% MB} & \textbf{Polarity quotient of:}  & \textbf{Source}\\

\midrule
2  & 4    & 7 & 57.14 & $\mathbb{G}_4(1,1)$ & \cite{bipartition_reversing_involutions_10}\\
3  & 15    & 22 & 68.18 & $\mathbb{G}_4(2,2)$ & \cite{delorme_french_polys_replication_85} \\
4 & 16   & 53   & 30.19 &  
$\mathbb{G}_4(1,1) \odot Opp(\mathbb{G}_4(1,1))^\star$ 
&\cite{delorme_opp_85}\\
\textbf{4} & \textbf{16}   & \textbf{53}   & \textbf{30.19} &  
$\mathbf{\mathbb{G}_4(1,1) \odot \mathbb{G}_4(1,1)}^\star$ 
&\textbf{this paper}\\
\textbf{6} & \textbf{60} & \textbf{187} & \textbf{32.09} & 
$\mathbf{\mathbb{G}_4(1,1)\odot \mathbb{G}_4(2,2)}$ & \textbf{this paper} \\
8 & 128   & 457 & 28.01 & $\mathbb{G}_4(1,3) \odot Opp(\mathbb{G}_4(1,3))$ &\cite{delorme_opp_85} \\
9 & 225   & 658 & 34.19 & $\mathbb{G}_4(2,2) \odot Opp(\mathbb{G}_4(2,2))^\star$ &\cite{delorme_opp_85}\\
\textbf{9} & \textbf{225}   & \textbf{658} & \textbf{34.19} & $\mathbf{\mathbb{G}_4(2,2) \odot \mathbb{G}_4(2,2)}^\star$ &\textbf{this paper}\\
9  & 585    & 658 & 88.91 & $\mathbb{G}_4(8,8)$ &\cite{delorme_french_polys_replication_85} \\
15 & 1215   & 3166 & 38.38 & $\mathbb{G}_4(2,4) \odot Opp(\mathbb{G}_4(2,4))$ &\cite{delorme_opp_85} \\
16 & 1600   & 3857 & 41.48 & $\mathbb{G}_4(3,3) \odot Opp(\mathbb{G}_4(3,3))$ &\cite{delorme_opp_85} \\
\textbf{18} & \textbf{2340} & \textbf{5527} & \textbf{42.34} & $\mathbf{\mathbb{G}_4(1,1)\odot \mathbb{G}_4(8,8)}$ & \textbf{this paper} \\
\textbf{19} & \textbf{2470} & \textbf{6518} & \textbf{37.90} &  $\mathbf{\mathbb{G}_4(1,1)\odot \mathbb{G}_4(8,8)},$ \textbf{replication} & \textbf{this paper} \\
\textbf{20} & \textbf{2600}   & \textbf{7621}   & \textbf{34.12} & $\mathbf{\mathbb{G}_4(1,1)\odot \mathbb{G}_4(8,8)},$ \textbf{replication} &  \textbf{this paper} \\
\bottomrule
\end{tabular}
}
\raggedright
\vspace{.5em}
\footnotesize
{

$^\star$ $\mathbb{G}_4(q,q) \cong Opp(\mathbb{G}_4(q,q))$, so $\mathbb{G}_4(q,q) \odot Opp(\mathbb{G}_4(q,q)) \cong \mathbb{G}_4(q,q) \odot \mathbb{G}_4(q,q)$. However, the polarity on an Opposite construction \cite{delorme_opp_85}  sends vertex $(a,b)$ to $(b,a),$ whereas our composed polarity sends $(a,b)$ to $(\rho(a), \rho(b))$. The polarities agree at $(a,b)$ only when $b=\rho(a).$ Corresponding polarity quotients seen here have the same degree, diameter and order, but are not isomorphic: their degree distributions differ. 
} 
\end{table}

\section{Conclusions}
 We have shown a structural compatibility between the Kronecker product and the polarity quotient on bipartite graphs admitting polarity, and using this, have found new families of graphs, two of which exist for any even diameter $2k$ using bipartite compatibility graphs. We have also demonstrated families of diameter 2, 3, and 5 using generalized polygons. Our degree 18, diameter 3 graph has the largest order as of October 2026.

The Moore bound is close to the best known bound for the degree-diameter problem, with better bounds differing by only $1$ or $2$, with a few special cases having slightly higher defects \cite{miller2012moore}. However, there are few known families of graphs that asymptotically approach the Moore bound. It is interesting to note that to our knowledge, all of the known families of graphs that asymptotically approach or meet the Moore bound have geometric constructions or interpretations, e.g., variants on the generalized polygons \cite{bachraty_siagiova_siran, bachraty_siran_2015, brown_1966, delorme_french_polys_replication_85, delorme_opp_85,erdosrenyi1962,  siagiova_siran}, and the new graphs derived from spherical buildings \cite{vanbatenburg2026asymptoticallyattainingmoorebound}.

There are immediate applications of this theoretical work. One area of interest is that of high-performance computer networks, where large low-diameter topologies have been proposed to meet the performance and scaling requirements for modern datacenters and supercomputers \cite{ besta2014slim,polarfly_sc22, PolarStar_23,lei2020bundlefly}. 
Such a network may be modeled as an undirected graph, with graph vertices as computational nodes, and graph edges as links between these nodes. Low diameter is desirable as it bounds worst-case hop count, and a large number of nodes gives higher compute performance to the system.
Successful approaches to the  degree-diameter problem speak directly to the design of optimal high-performance computing networks.

\section{Tool and Computational Resource Disclosure}
GPT-5.6 and -6.1 were used for some of the literature searches and for brainstorming about the bipartite compatibility graph polarities of  Section~\ref{sec:bipartite_compatibility}, and also for an editing pass-through upon completion to correct minor errors and typos. Figure~\ref{fig:bcg_compare} was generated by GPT-6.1.

The polarity discussion in Section~\ref{sec:bipartite_compatibility} was directed by the authors, who noted that bipartite compatibility graphs do have polarities having the needed properties, and was guided by an inquiry into the nature of these polarities.
The authors developed the work into the two polarity constructions in this section,  and formulated and verified all mathematical arguments. 

The authors take full responsibility
for all content.
\section{Acknowledgements}This research was supported in part by the Picker Interdisciplinary Science Institute at Colgate University,
and also by the U.S. Department of Energy through Los Alamos National Laboratory, operated by Triad National Security, LLC, for the U.S. DOE (Contract no. 89233218CNA000001). The U.S. Government retains an irrevocable, nonexclusive, royalty-free license to publish, translate, reproduce, use, or dispose of the published form of this work and to authorize others to do the same.
This paper has been assigned the LANL identification number LA-UR-26-28762.

\clearpage
\bibliographystyle{plain}
\bibliography{references}

@article{siagiova_siran,
title = {Approaching the {M}oore bound for diameter two by {C}ayley graphs},
journal = {Journal of Combinatorial Theory, Series B},
volume = {102},
number = {2},
pages = {470-473},
year = {2012},
issn = {0095-8956},
doi = {https://doi.org/10.1016/j.jctb.2011.07.005},
url = {https://www.sciencedirect.com/science/article/pii/S0095895611000736},
author = {J. \v{S}iagiov{\'a} and J. \v{S}ir{\'a}\v{n}}
}

@article{bachraty_siagiova_siran,
title = {Asymptotically approaching the {M}oore bound for diameter three by {C}ayley graphs},
journal = {Journal of Combinatorial Theory, Series B},
volume = {134},
pages = {203-217},
year = {2019},
issn = {0095-8956},
doi = {https://doi.org/10.1016/j.jctb.2018.06.003},
url = {https://www.sciencedirect.com/science/article/pii/S0095895618300509},
author = {M. Bachrat\'y and J. \v{S}iagiov{\'a} and J. \v{S}ir{\'a}\v{n}}
}

@book{AbramenkoBrown2008,
  author    = {Abramenko, P. and Brown, K. S.},
  title     = {Buildings: Theory and Applications},
  series    = {Graduate Texts in Mathematics},
  volume    = {248},
  publisher = {Springer},
  year      = {2008},
  doi       = {10.1007/978-0-387-78835-7}
}

@article{AlonChungGraham1994,
  author  = {Alon, N. and Chung, F. R. K. and Graham, R. L.},
  title   = {Routing Permutations on Graphs via Matchings},
  journal = {SIAM J. Discrete Math.},
  volume  = {7},
  year    = {1994},
  pages   = {513--530},
  doi     = {10.1137/S0895480192236628}
}

@book{StanleyEnumComb_2012,
  author    = {R. P. Stanley},
  title     = {Enumerative Combinatorics, Volume 1},
  edition   = {second},
  series    = {Cambridge Studies in Advanced Mathematics},
  number    = {49},
  publisher = {Cambridge University Press},
  year      = {2012},
  doi       = {10.1017/CBO9781139058520}}

@book{Taylor1992,
  author    = {D. E. Taylor},
  title     = {The Geometry of the Classical Groups},
  series    = {Sigma Series in Pure Mathematics},
  volume    = {9},
  publisher = {Heldermann Verlag},
  address   = {Berlin},
  year      = {1992},
  isbn      = {3-88538-009-9}
}

@book{Bollobas1978,
  author    = {Bollob{\'a}s, B.},
  title     = {Extremal Graph Theory},
  series    = {London Mathematical Society Monographs},
  volume    = {11},
  publisher = {Academic Press},
  address   = {New York},
  year      = {1978}
}

@article{CambieCamesVanBatenburgDeJoannisDeVerclosKang2022,
  author  = {Cambie, S. and Cames van Batenburg, W. and
             de Joannis de Verclos, R. and Kang, R. J.},
  title   = {Maximizing Line Subgraphs of Diameter at Most {$t$}},
  journal = {SIAM J. Discrete Math.},
  volume  = {36},
  number  = {2},
  pages   = {939--950},
  year    = {2022},
  doi     = {10.1137/21M1437354}
}

@misc{vanbatenburg2026asymptoticallyattainingmoorebound,
  author       = {Cames van Batenburg, W. and Korsky, S.},
  title        = {Asymptotically attaining the {Moore} bound},
  year         = {2026},
  howpublished = {arXiv:2608.03965v1},
  doi          = {10.48550/arXiv.2608.03965v1},
  note         = {\url{https://doi.org/10.48550/arXiv.2608.03965}}
}

@book{gratzer_universal_alg,
  title={Universal Algebra},
  author={Gr{\"a}tzer, G.},
  isbn={9780387903552},
  edition = {2nd},
  url={https://books.google.com/books?id=YrY3xgEACAAJ},
  year={1979},
  address={New York, NY, USA},
  publisher={Springer New York}
}

@incollection{diam_kronecker_orig,
  author    = {Hell, P.},
  title     = {Subdirect products of bipartite graphs},
  booktitle = {Infinite and Finite Sets, {Vol. II}, \textnormal{A. Hajnal, R. Rado, and V. T. S{\'o}s, eds.}},
  series    = {Colloq. Math. Soc. J. Bolyai},
  volume    = {10},
  pages     = {857--866},
  publisher = {North-Holland},
  address   = {Amsterdam},
  year      = {1975}
}

@article{diam_kronecker,
author = {Pu\v{s}, V.},
journal = {Comment. Math. Univ. Carolin.},
language = {eng},
number = {2},
pages = {233-239},
publisher = {Charles University in Prague, Faculty of Mathematics and Physics},
title = {A remark on distances in products of graphs},
url = {http://eudml.org/doc/17539},
volume = {28},
year = {1987},
}

@article{bipartition_reversing_involutions_08,
title = {On direct product cancellation of graphs},
journal = {Discrete Math.},
volume = {309},
number = {8},
pages = {2538-2543},
year = {2009},
issn = {0012-365X},
doi = {https://doi.org/10.1016/j.disc.2008.06.004},
url = {https://www.sciencedirect.com/science/article/pii/S0012365X08003889},
author = {R. Hammack}
}

@article{bipartition_reversing_involutions_10,
author = {Abay-Asmerom, G. and Hammack, R. and Larson, C. and Taylor, D.},
year = {2010},
pages = {2042-2052},
title = {Direct Product Factorization of Bipartite Graphs with Bipartition-reversing Involutions},
volume = {23(4)},
journal = {SIAM J. Discrete Math.},
doi = {10.1137/090751761}
}

@article{lw_polarity_graphs,
author = {Lazebnik, F. and Woldar, A.},
year = {2001},
pages = {65-86},
title = {General properties of some families of graphs defined by systems of equations},
volume = {38},
journal = {J. Graph Theory},
doi = {10.1002/jgt.1024}
}

@article{Birkhoff1944,
  author  = {G. Birkhoff},
  title   = {Subdirect Unions in Universal Algebra},
  journal = {Bull. Amer. Math. Soc.},
  volume  = {50},
  number  = {10},
  year    = {1944},
  pages   = {764--768},
  doi     = {10.1090/S0002-9904-1944-08235-9}
}

@article{congruences_no_loop,
author = {I. Broere and J. Heidema and L. Pretorius},
title = {Graph congruences and what they connote},
journal = {Quaest. Math.},
volume = {41},
number = {8},
pages = {1045--1059},
year = {2018},
publisher = {Taylor \& Francis},
doi = {10.2989/16073606.2017.1419388},
URL = { 
    
        https://doi.org/10.2989/16073606.2017.1419388
},
eprint = {     
        https://doi.org/10.2989/16073606.2017.1419388
}
}

@article{congruences_loop_1,
author = {Veldsman, S.},
year = {2020},
pages = {123-132},
title = {Congruences and subdirect representations of graphs},
volume = {8},
journal = {Electron. J. Graph
Theory Appl.
},
doi = {10.5614/ejgta.2020.8.1.9}
}

@article{congruences_loop_2,
author = {Broere, I. and Heidema, J. and Veldsman, S.},
year = {2020},
pages = {1067--1084},
title = {Congruences and {H}oehnke radicals on graphs},
volume = {40},
journal = {Discuss. Math. Graph Theory},
doi = {10.7151/dmgt.2166}
}

@book{hammack2011handbook,
  title={Handbook of Product Graphs},
  author={Hammack, R. and Imrich, W. and Klav\v{z}ar, S.},
  isbn={9781439813058},
  series={Discrete Mathematics and Its Applications},
  edition = {2nd},
  address = {Boca Raton, FL},
  url={https://books.google.com/books?id=WiB6UO1nqHAC},
  year={2011},
  publisher={CRC Press}
}

@article{gen_poly_orig_tits,
author = {Tits, J.},
year = {1959},
pages = {13–60},
title = {Sur la trialit\'e et certains groupes qui s’en d\'eduisent},
volume = {2},
journal = {Publ. Math. Inst. Hautes
\'Etudes Sci. },
doi = {https://doi.org/10.1007/BF02684706}
}

@incollection{THAS_gen_polys,
title = {{G}eneralized {P}olygons},
editor = {F. Buekenhout},
booktitle = {Handbook of Incidence Geometry},
publisher = {North-Holland},
address = {Amsterdam},
pages = {383-431},
year = {1995},
isbn = {978-0-444-88355-1},
doi = {https://doi.org/10.1016/B978-044488355-1/50011-6},
url = {https://www.sciencedirect.com/science/article/pii/B9780444883551500116},
author = {J. A. Thas}
}

@book{GenPolys_VanMaldgehem,
  title={Generalized Polygons},
  author={Van Maldeghem, H.},
  isbn={9783034802703},
  lccn={2011941498},
  series={Modern Birkh{\"a}user Classics},
  url={https://books.google.com/books?id=S03hohjz7n4C},
  year={2012},
  publisher={Springer Basel}
}

@ARTICLE{lozpineda2010,
  author={Loz, E. and Pineda-Villavicencio, G.},
  journal={Comput. J.}, 
  title={New Benchmarks for Large-Scale Networks with Given Maximum Degree and Diameter}, 
  year={2010},
  volume={53},
  number={7},
  pages={1092-1105},
  doi={10.1093/comjnl/bxp091}}

@article{feit_higman,
title = {The nonexistence of certain generalized polygons},
journal = {J. Algebra},
volume = {1},
number = {2},
pages = {114-131},
year = {1964},
issn = {0021-8693},
doi = {https://doi.org/10.1016/0021-8693(64)90028-6},
url = {https://www.sciencedirect.com/science/article/pii/0021869364900286},
author = {W. Feit and G. Higman}
}

@article{iso_comp_kronecker_bipartite_2006,
title = {Proof of a conjecture concerning the direct product of bipartite graphs},
journal = {European J. Combin.},
volume = {30},
number = {5},
pages = {1114-1118},
year = {2009},
note = {Special Issue on Metric Graph Theory},
issn = {0195-6698},
doi = {https://doi.org/10.1016/j.ejc.2008.09.015},
url = {https://www.sciencedirect.com/science/article/pii/S0195669808001820},
author = {R. Hammack}
}

@article{iso_comp_kronecker_bipartite_1997,
  title={Isomorphic components of {K}ronecker product of bipartite graphs},
  author={P. K. Jha and S. Klav\v{z}ar and B. Zmazek},
  journal={Discuss. Math. Graph Theory},
  year={1997},
  volume={17},
  pages={301-309},
  url={https://api.semanticscholar.org/CorpusID:2968557}
}

@article{kronecker_1962,
author = {Weichsel, P.},
year = {1962},
pages = {47--52},
title = {The {K}ronecker Product of Graphs},
volume = {13},
journal = {Proc. Amer. Math. Soc. },
doi = {10.2307/2033769}
}

@misc{Comellas2026Dataset,
  author       = {Comellas, F.},
  title        = {Table of Large Degree/Diameter Graphs},
  year         = {2026},
  howpublished = {Mendeley Data, Version 11},
  doi          = {10.17632/d75dzbjd4k.11},
  note         = {DOI: 10.17632/d75dzbjd4k.11},
  url          = {https://data.mendeley.com/datasets/d75dzbjd4k/11}
}

@article{deg_diam_pub,
  title={New results on the degree-diameter problem for undirected graphs},
  author={F. Comellas},
  journal={Electron. J. Graph Theory Appl.},
  year={2025},
  number = {1},
  volume={13},
  pages={211-215},
}

@misc{deg_diam_table_2026,
  author = {F. Comellas},
  title = {Degree--Diameter Table of Large Graphs},
  howpublished = {\url{https://web.mat.upc.edu/francesc.comellas/delta-d/taula_delta_d.html}},
  note = {Online resource maintained by the author. Accessed 20 Oct 2026},
  year = {2026}
}

@inproceedings{PolarStar_23,
      title={Polar{S}tar: Expanding the Horizon of Diameter-3 Networks}, 
      author={K. Lakhotia and L. Monroe and K. Isham and M. Besta and N. Blach and T. Hoefler and F. Petrini},
      year={2024},
pages ={345-357},
    booktitle = {{Proceedings of the 36th ACM Symposium on Parallelism in Algorithms and Architectures}},
    location = {Nantes, France},
    series = {SPAA '24}
}

@book{paynethas,
  author       = {Payne, S. E. and Thas, J. A.},
  isbn         = {{978-3-03719-066-1}},
  language     = {{eng}},
  pages        = {{xi + 287}},
  publisher    = {{European Mathematical Society}},
  edition={2nd},
  title        = {{Finite Generalized Quadrangles}},
  year         = {{2009}},
}

@inproceedings{polarfly_sc22,
  title={Polar{F}ly: A Cost-Effective and Flexible
Low-Diameter Topology},
  author={Lakhotia, K. and Besta, M. and Monroe, L. and Isham, K. and Iff, P. and Hoefler, T. and Petrini, F.},
  booktitle={{Proceedings of the International Conference on High Performance Computing, Networking, Storage and Analysis (SC22)}},
  pages={1--15},
  year={2022},
  organization={IEEE}
}

@inproceedings{lei2020bundlefly,
  title={Bundlefly: A low-diameter topology for multicore fiber},
  author={Lei, F. and Dong, D. and Liao, X. and Duato, J.},
  booktitle={Proceedings of the 34th ACM International Conference on Supercomputing (ICS '20)},
  pages = {1--11},
doi   = {10.1145/3392717.3392747},
  year={2020}
}

@inproceedings{besta2014slim,
  title={Slim {F}ly: A cost effective low-diameter network topology},
  author={Besta, M. and Hoefler, T.},
  booktitle={Proceedings of the International Conference for High Performance Computing, Networking, Storage and Analysis (SC14)},
  pages={348--359},
  year={2014},
  organization={IEEE}
}

@article{delorme_french_polys_replication_85,
    title={Grands Graphes de Degr\'{e} et Diam\`{e}tre Donn\'{e}s},
    author={C. Delorme},
    journal={European J. Combin.},
    volume={6},
    pages={291--302},
    year={1985}
}

@article{Bannai1973OnFM,
  title={On finite {M}oore graphs},
  author={E. Bannai and T. Ito},
  journal={J. Fac. Sci. Univ. Tokyo Sect. IA Math.},
  year={1973},
  volume={20},
  pages={191-208}
}

@article{Damerell1973OnMG,
  title={On {M}oore graphs},
  author={R. M. Damerell},
  journal={Proc. Cambridge Philos. Soc.},
  volume={74},
  pages={227-236},
  year={1973}
}

@ARTICLE{hoffmansingleton1960,
  author={Hoffman, A. J. and Singleton, R. R.},
  journal={IBM J. Res. Develop.}, 
  title={On {M}oore Graphs with Diameters 2 and 3}, 
  year={1960},
  volume={4},
  number={5},
  pages={497-504},
  doi={10.1147/rd.45.0497}
}

@article{delorme_opp_85,
  title={Large bipartite graphs with given degree and diameter},
  author={Delorme, C.},
  journal={J. Graph Theory},
  volume={9},
  number={3},
  pages={325--334},
  year={1985},
  publisher={Wiley Online Library}
}

@misc{comb_wiki_degdiam_general,
author = {Loz, E. and P\'erez-Ros\'es, H. and G. Pineda-Villavicencio},
title  = {The Degree-Diameter Problem for General Graphs},
year   = {2022},
howpublished    = {\url{http://www.combinatoricswiki.org/wiki/The_Degree_Diameter_Problem_for_General_Graphs}},
    note = {Accessed 20 May 2026}

}

@article{miller2012moore,      title={Moore Graphs and Beyond: A survey of the Degree/Diameter Problem},
    volume={Dynamic Surveys, \#DS14}, 
    url={https://www.combinatorics.org/ojs/index.php/eljc/article/view/DS14}, 
    note = {2nd edition, DOI: 10.37236/35}, 
    journal={Electron. J. Combin.}, author={Miller, M. and \v{S}ir\'{a}\v{n}, J.}, 
    year={2013} 
}

@article{erdosrenyi1962,
author = {Erd{\H o}s, P. and R\'enyi, A.},
year = {1962},
pages = {623-641},
title = {On a problem in the theory of graphs},
volume = {7A},
journal = {Publ. Math. Inst. Hungar. Acad. Sci.}
}

@article{brown_1966, 
title={On Graphs that do not Contain a {T}homsen Graph}, 
volume={9}, 
DOI={10.4153/CMB-1966-036-2}, number={3}, journal={Canad. Math. Bull.}, publisher={Cambridge University Press}, author={Brown, W. G.}, year={1966}, pages={281–285}
}

@article{bachraty_siran_2015,
author = {Bachrat\'y, M. and \v{S}ir{\'a}\v{n}, J.},
year = {2015},
month = {Jan},
pages = {55-67},
title = {Polarity graphs revisited},
volume = {8},
journal = {Ars Mathematica Contemporanea},
doi = {10.26493/1855-3974.527.74e}
}

\end{document}